\documentclass{article}
\usepackage{graphicx} % Required for inserting images
\usepackage{hyperref}
\usepackage{amsgen,amsmath,amstext,amsbsy,amsopn,amssymb}
\usepackage{amsmath}
\usepackage{multirow}
\usepackage{xspace}
\usepackage{epstopdf}
\usepackage{authblk}
\usepackage{xcolor}
\usepackage{xr-hyper}
\usepackage{hyperref}
\usepackage{natbib}
\usepackage[english]{babel}
\usepackage{amsthm}
\usepackage{geometry}
\usepackage{accents}
\usepackage{ulem}
\usepackage{algorithm}
\usepackage{algpseudocode}
\usepackage{siunitx}
\newcommand{\methodname}[1]{Advanced SSS}
\usepackage[toc,page]{appendix} 
\def\ha{\hat{a}}
\newtheorem{theorem}{Theorem}%[section]
\newtheorem{corollary}{Corollary}[theorem]

\title{Advanced Distribution Theory for Significance in Scale Space}
\date{}
\author[1]{Rui Liu }
\author[1]{Jan Hannig}
\author[1,2]{J. S. Marron}

\affil[1]{Department of Statistics and Operations Research,  The University of North Carolina at Chapel Hill}

\affil[2]{School of Data Science and Society,  The University of North Carolina at Chapel Hill}

\begin{document}

\maketitle
\begin{abstract}
	Smoothing methods find signals in noisy data. A challenge for Statistical inference is the choice of smoothing parameter. SiZer addressed this challenge in one-dimension by detecting significant slopes across multiple scales, but was not a completely valid testing procedure. This was addressed by the development of an advanced distribution theory that ensures fully valid inference in the 1-D setting by applying extreme value theory. A two-dimensional extension of SiZer, known as Significance in Scale Space (SSS), was developed for image data, enabling the detection of both slopes and curvatures across multiple spatial scales. However, fully valid inference for 2-D SSS has remained unavailable, largely due to the more complex dependence structure of random fields. In this paper, we use a completely different probability methodology which gives an advanced distribution theory for SSS, establishing a valid hypothesis testing procedure for both slope and curvature detection. When applied to pure noise images (no true underlying signal), the proposed method  controls the Type I error, whereas the original SSS identifies spurious features across scales. When signal is present, the proposed method maintains a  high level of statistical power, successfully identifying important true slopes and curvatures in real data such as gamma camera images.
	
\end{abstract}

\section{Introduction}

In smoothing based statistical analysis identifying meaningful structure in complex data is challenging, particularly when features arise at multiple scales (smoothing parameters) and across varying dimensions \citep{lindeberg2013scale}.  The introduction of the SiZer (SIgnificant ZERo crossings) method by \citet{chaudhuri1999sizer} marked a pivotal advancement in addressing this challenge by offering a framework for multiscale smoothing analysis. SiZer, designed for the analysis of one-dimensional data, is based on statistical significance of derivatives of smoothed functions, allowing for the identification of significant features while controlling for spurious noise across multiple scales. Through its visual representation via color-coded maps, SiZer enabled intuitive understanding of the underlying data structure, making it a practical tool for exploratory data analysis \citep{chaudhuri2000scale}. Beyond kernel based implementations, extensions of SiZer include smoothing spline SiZer \citep{zhang2004simple, marron2005sizer}, local likelihood SiZer \citep{li2005local}, SiZer maps for additive models \citep{gonzalez2008sizer, martinez2005sizer}, SiZer for time series \citep{park2004dependent,rondonotti2007sizer,park2009improved,park2009sizer}, and circular SiZer for directional data \citep{oliveira2014circsizer,huckemann2016circular}. 

While SiZer led the way in scale space based statistical analysis of 1-d signals \citep{chaudhuri1999sizer}, it does not supply a fully valid testing procedure. \citet{hannig2006advanced}  developed an advanced distribution theory for 1-d SiZer that established valid size control based on the extreme value theory of \citet{hsing1996extremes}. 
The 2-d analog of SiZer, known as Significance in Scale Space (SSS), was introduced by \citet{ godtliebsen2002significance_cluster,godtliebsen2002significance, godtliebsen2004statistical}. The goal of SSS is to detect surface features such as peaks, valleys, and ridges in images.  However, the multiple comparison adjustment used there was a 2-d version of that used in \citet{chaudhuri1999sizer}, which was similarly only approximately valid.  
Fully valid statistical inference for 2-d SSS is particularly challenging because the dependence structure of random fields is much more complex than for 1-d random processes. In this paper, valid inference is provided using results from the extreme value theory for random fields of \citet{french2013asymptotic}.

Figure~\ref{fig:pure_noise} demonstrates the value of the proposed advanced SSS by comparing it with the classical SSS~\citep{godtliebsen2004statistical}. Both methods are applied to a  pure-noise image (no true underlying signal) with i.i.d. \(Y_{i,j}\sim N(0,1)\) on a \(64\times 64\) lattice, shown in the left panel of Figure~\ref{fig:pure_noise}. Advanced SSS (top row) reports no significant structure at any location for any  bandwidth \(h\in\{2,4,8,16\}\), which is consistent with the absence of signal by design.  In contrast, the classical SSS flags spurious features at \(h=2,8,\) and \(16\). This demonstrates the correct handling of type I error by the advanced SSS. 
To complement this specificity check, we also present images that contain genuine structure in Figures~\ref{fig:curvature_comparison} and~\ref{fig:gamma}, in Sections~\ref{sec_simulation} and~\ref{sec_realdata}. While  the advanced SSS flags fewer significant pixels, there is essentially no loss in underlying structure recovery in those examples. 
\begin{figure}[htbp]
	\center{\includegraphics[width=18cm]  {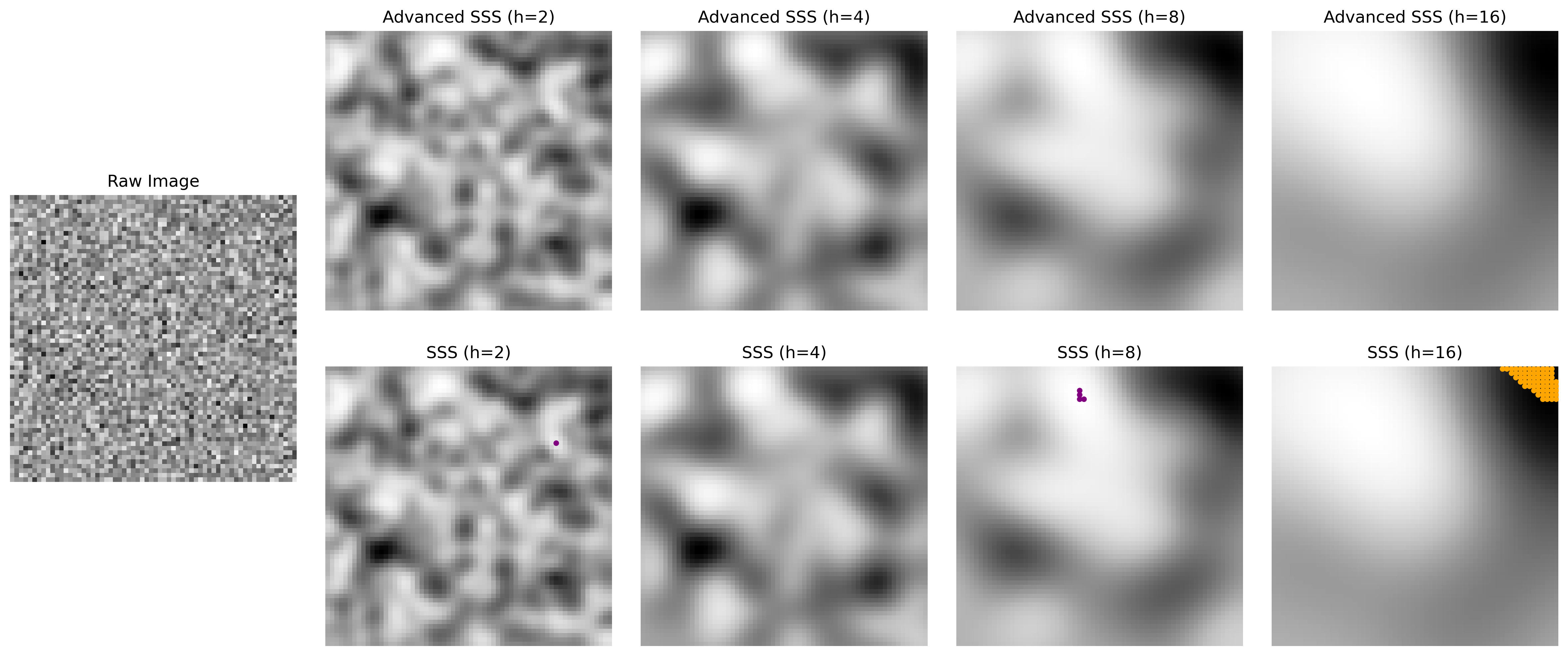}}
	\caption{
		Comparison of curvature identification results for a pure noise example from the proposed advanced SSS method (top row) and the original SSS method (bottom row) across four bandwidths \( h = 2, 4, 8, 16 \). The raw image is shown in the leftmost panel. Each colored dot represents a statistically significant curvature classification:  purple for ridge and orange for valley structures. Note the original SSS indicates spurious structure for three of the four scales.
	}
	\label{fig:pure_noise}
\end{figure} 

Applications of SiZer and SSS include ecology, environmental science, geoscience, econometrics, biomedicine, and genomics. For example, \citet{sonderegger2009using} and \citet{daily2012experimental} detected ecological thresholds and \citet{clements2010use} assessed recovery in ecosystems; \citet{ryden2010exploring} analyzed hurricane trends; \citet{rudge2008finding} identified peaks in geochemical distributions; \citet{zambom2013review} reviewed econometric uses; \citet{harezlak2020ps} investigated body weight profiles of HIV infected patients;  and \citet{liu2025significance} adapted SSS to Hi-C contact maps.

The rest of this paper is organized as follow. Section~\ref{sec-method} derives the advanced distribution theory for SSS. Section~\ref{sec_simulation} demonstrates statistical properties using simulations. Section~\ref{sec_realdata} presents gamma camera image analysis. Section~\ref{sec_discussion} concludes. Proofs can be found in the Appendix and additional simulations are in the supplementary materials.

\section{Field-wise Extreme Value Theory for SSS}
\label{sec-method}

To develop the distributional properties of Curvature SSS \citep{godtliebsen2004statistical}, we first establish a foundation in kernel smoothing techniques. The objective is to derive rigorous statistical inference for the curvature-based significance analysis in scale space by leveraging field-wise extreme value theory.

Assume the plane's size is $n\times n$, our notation for regression data is $(i,j,Y_{i,j})$ for integer image coordinates $i,j$, respectively. This can also be viewed in terms of in-fill asymptotics as \begin{equation}
	\label{eq_xij}
	{\bf x_{i,j}}=(x_{i,1},x_{j,2})=(i\Delta,j\Delta),
\end{equation}
where $\Delta\rightarrow 0$. Such data can originate from multiple sources and can be represented using various mathematical models. Suppose $Y_{i,j}=m({\bf x_{i,j}})+\epsilon_{i,j}$, where $m$ is the regression
function, and where $\epsilon_{i,j}$ are independent and identically distributed with mean 0 and variance $\sigma^2$.  Define at each location ${\bf x_{0}}=(i_0\Delta,j_0\Delta)$, an estimate of $m({\bf x_0})$ as
\begin{equation*}
	\begin{aligned}
		\hat{m}({\bf x_0}) = \tilde{a}_{00}, 
	\end{aligned}
\end{equation*}
where 
\begin{equation}
	\label{argmin}
	\begin{aligned}
		(\tilde{a}_{00},\tilde{a}_{10},\tilde{a}_{01},\tilde{a}_{20},\tilde{a}_{11},\tilde{a}_{02})
		=&\mathop{\arg\min}_{a_{00},a_{10},a_{01},a_{20},a_{11},a_{02}}\sum_{i=1}^{n}\sum_{j=1}^{n}\{Y_{i,j}-[a_{00}+a_{10}(i-i_0)+a_{01}(j-j_0)\\
		&+a_{20}(i-i_0)^2+2a_{11}(i-i_0)(j-j_0)+a_{02}(j-j_0)^2]\}^2K_{\tilde{h}}(i-i_0,j-j_0),\\
	\end{aligned}
\end{equation}
where $K_{\tilde{h}}$ is the spherically symmetric Gaussian density function on the 2-dimensional plane, with standard deviation $\tilde{h}$. The Gaussian kernel is chosen due to its well-established theoretical properties, particularly its ability to provide a smooth and continuous estimate  as discussed in \cite{chaudhuri1999sizer, chaudhuri2000scale}. For example, in one dimension, it possesses the variation diminishing property,  as $h$ increases, the count of bumps decreases or remains non-increasing, ensuring a controlled and interpretable smoothing effect.	

In the following, for notational convenience, we assume \eqref{eq_xij}. The random design regression case can be treated analogously. Set  $h=\tilde{h}/\Delta$ and introduce the notation:
\begin{equation*}
	G_{mk}^h = \sum_{i=1}^{n}\sum_{j=1}^{n} (i-i_0)^m (j-j_0)^k K_h(i-i_0,j-j_0), \quad H_{mk}^h = \sum_{i=1}^{n}\sum_{j=1}^{n} Y_{i,j} (i-i_0)^m (j-j_0)^k K_h(i-i_0,j-j_0).%K_h(\uxij,x_0).
\end{equation*}

Let us define $g$ to be the double sum minimized in \eqref{argmin}. To obtain the optimal estimates of the parameters, we differentiate the objective function with respect to each coefficient. This leads to a system of linear equations in $a_{00},a_{10},a_{01},a_{20},a_{11},a_{02}$: 
\begin{equation}
	\label{eq-deriv}
	%\left\{
	\begin{aligned}
		0 &= \frac{\mathrm{d} g}{\mathrm{d}a_{00}} = H^h_{00} - a_{00}G^h_{00} - a_{10}G^h_{10} - a_{01}G^h_{01} - a_{20}G^h_{20} - a_{11}G^h_{11} - a_{02}G^h_{02}, \\
		0 &= \frac{\mathrm{d} g}{\mathrm{d}a_{10}} = H^h_{10} - a_{00}G^h_{10} - a_{10}G^h_{20} - a_{01}G^h_{11} - a_{20}G^h_{30} - a_{11}G^h_{21} - a_{02}G^h_{12}, \\
		0 &= \frac{\mathrm{d} g}{\mathrm{d}a_{01}} = H^h_{01} - a_{00}G^h_{01} - a_{10}G^h_{11} - a_{01}G^h_{02} - a_{20}G^h_{21} - a_{11}G^h_{12} - a_{02}G^h_{03}, \\
		0 &= \frac{\mathrm{d} g}{\mathrm{d}a_{20}} = H^h_{20} - a_{00}G^h_{20} - a_{10}G^h_{30} - a_{01}G^h_{21} - a_{20}G^h_{40} - a_{11}G^h_{31} - a_{02}G^h_{22}, \\
		0 &= \frac{\mathrm{d} g}{\mathrm{d}a_{11}} = H^h_{11} - a_{00}G^h_{11} - a_{10}G^h_{21} - a_{01}G^h_{12} - a_{20}G^h_{31} - a_{11}G^h_{22} - a_{02}G^h_{13}, \\
		0 &= \frac{\mathrm{d} g}{\mathrm{d}a_{02}} = H^h_{02} - a_{00}G^h_{02} - a_{10}G^h_{12} - a_{01}G^h_{03} - a_{20}G^h_{22} - a_{11}G^h_{13} - a_{02}G^h_{04}.
	\end{aligned}
	%\right.
\end{equation}

Due to the symmetry of the kernel function, when either $m$ or $k$ is odd, we approximate $G^h_{mk} \approx 0$ as $n \to \infty$.
A useful integral approximation is:
\begin{equation*}
	%\left\{
	\begin{aligned}
		G^h_{20} &= \sum_{i=1}^{n}\sum_{j=1}^{n}(i-i_0)^2K_h(i-i_0,j-j_0) \\
		&\approx \int_{-\infty}^{\infty}\int_{-\infty}^{\infty}(w-i_0)^2\frac{1}{2\pi h^2}\exp\left(-\frac{(w-i_0)^2+(z-j_0)^2}{2h^2}\right)dwdz \\
		&= \int_{-\infty}^{\infty}(w-i_0)^2\frac{1}{\sqrt{2\pi h^2}}\exp\left(-\frac{(w-i_0)^2}{2h^2}\right)dw=h^2. 
	\end{aligned}
	%\right.
\end{equation*}
Similar integral approximations lead to  $G^h_{00} \approx 1, 
G^h_{20}=  G^h_{02}\approx h^2,
G^h_{22} \approx h^4, G^h_{40} = G^h_{04} \approx 3h^4.$ 
By substituting these approximations into \eqref{eq-deriv}, we define  estimates for the first and second derivatives: 
\begin{equation*}
	\left\{
	\begin{aligned}
		\ha_{10} &= \frac{H_{10}}{h^2}, \quad \ha_{01} = \frac{H_{01}}{h^2}, \\
		\ha_{11} &= \frac{H_{11}}{2h^4}, \quad \ha_{20} = \frac{H_{20} - H_{00}h^2}{2h^4}, \quad \ha_{02} = \frac{H_{02} - H_{00}h^2}{2h^4}.
	\end{aligned}
	\right.
\end{equation*}

These estimators provide the basis for the extreme value analysis of field-wise slope and curvature in scale space. The following subsections will present the theorems governing the statistical significance of first and second derivative estimates. The gradient (vector of first derivatives) estimate is crucial for identifying edges and directional changes in the field, allowing for the detection of significant structural features. The Hessian (matrix of second derivatives) plays a key role in characterizing curvature, distinguishing between local maxima, minima, and saddle points. Together, these derivatives enable a comprehensive understanding of the underlying geometric properties of the observed field, facilitating robust  significance assessment.

The foundation of our analysis is a two-dimensional extension of Theorem 2.2 in \citet{hsing1996extremes}, which played a crucial role in developing the advanced distribution theory for one-dimensional SiZer in \cite{hannig2006advanced}. To extend these results to two dimensions, we leverage Theorem 1 from \citet{french2013asymptotic}, which builds on the framework of \citet{hsing1996extremes} by incorporating conditions for spatial dependence and deriving limiting distributions. 

Let $\xi_{n,i,j}$, $i,j=0,1,2,\ldots,n$ for $n=1,2,\ldots$ be a sequence of stationary (in $i,j$) random fields indexed by $n$. Assume that 
\begin{equation*}
	E\xi_{n,i,j}=0\quad\text{and}\quad E\xi_{n,i,j}^2=1.
\end{equation*}
Writing $\rho_{n,i,j}=E(\xi_{n,l,k}\xi_{n,l+i,k+j})$, let $\delta_{0,0}=0$ and for $i,j \in \mathbb{Z}\setminus (0,0)$ assume that
\begin{equation}
	\label{2d_assp1.2}
	\delta_{i,j}=\lim_{n\rightarrow\infty}(1-\rho_{n,i,j})\log n \in(0,\infty].
\end{equation}
For every $x\in(-\infty,\infty)$, let
$u_n(x)=x/a_n+b_n,$
where 
$
a_n=(2\log n)^{1/2}
$
and
\begin{equation*}
	b_n=(2\log n)^{1/2}-\frac{1}{2}(2\log n)^{-1/2}(\log\log n+\log4\pi).
\end{equation*}
It is well-known, e.g. from \citep{leadbetter2012extremes} that 
\begin{equation*}
	\label{un_approx}
	\lim_{n\rightarrow\infty}n[1-\Phi(u_n(x))]=e^{-x}
\end{equation*}

Next we introduce Theorem 1 in \citet{french2013asymptotic}, which is a 2-dimensional extension of Theorem  2.2 in \citet{hsing1996extremes}. 
\begin{theorem}[\citet{french2013asymptotic}]
	\label{2d_thm2.2}
	Assume that (\ref{2d_assp1.2}) holds. Also assume that there  exists a sequence of positive integers $l_n$ such that 
	\begin{equation}
		\label{2d_2.1}
		\frac{l_n}{n}\rightarrow0,
	\end{equation}
	and for which 
	\begin{equation}
		\label{2d_2.5}
		\lim_{n \to \infty} \sup_{\sqrt{i^2+j^2} \geq l_n} |\rho_{n,i,j}| \log n = 0,
	\end{equation}
	and
	\begin{equation}
		\label{2d_2.6}	\lim_{m \to \infty} \lim_{n \to \infty} \sup \sum_{(i,j) \in \{0,1,...,l_n\}^2 \setminus \{0,1,...,m\}^2} n^{-2\frac{1-\rho_{n,i,j}}{1+\rho_{n,i,j}}} \frac{\left( \log n \right)^{-\frac{\rho_{n,i,j}}{1+\rho_{n,i,j}}}}{\sqrt{1-\rho^2_{n,i,j}}}   = 0.
	\end{equation}
	Then
	\begin{equation}
		\label{2d_lim}
		\lim_{n\rightarrow\infty}P(\max_{1\leq i,j\leq n}\xi_{n,i,j}\leq u_{n^2}(x))=\exp(-\vartheta\exp(-x)),\quad -\infty<x<\infty,
	\end{equation}
	where
	\begin{equation*}
		\vartheta=P(E/4+\sqrt{\frac{1}{2}\delta_{i,j}}W_{i,j}\leq\delta_{i,j}\text{ for all }(i,j) \in \{\mathbb{N} \times \{0\}\} \cup \{\mathbb{Z} \times \mathbb{N}\}),
	\end{equation*}
	where $E$ is a standard exponential random variable independent of $W_{i,j}$ and $\{W_{i,j}\}$ have a jointly normal distribution with mean 0, variance 1 and covariance
	\begin{equation}
		\begin{aligned}
			\label{2d_2.4}
			EW_{i_1,j_1}W_{i_2,j_2}=\frac{\delta_{i_1,j_1}+\delta_{i_2,j_2}-\delta_{i_1-i_2,j_1-j_2}}{2(\delta_{i_1,j_1}\delta_{i_2,j_2})^{1/2}}.
		\end{aligned}
	\end{equation}
\end{theorem}

In what  follows, we need to verify Conditions \eqref{2d_assp1.2}-\eqref{2d_2.6}.

\subsection{Significance of Slopes}
\label{sec-extreSSS_1st}
Next, we developed limiting distributions for assessing field-wise statistical significance of the gradient estimate.
At any location ${\bf x_0}=(i_0\Delta,j_0\Delta)$, consider the estimated directional derivative
\begin{equation}
	\label{direct_1st}
	\begin{aligned}
		&(u,v)\left(
		\begin{matrix}
			\ha_{10}\\
			\ha_{01}
		\end{matrix}
		\right)=\ha_{10}u+\ha_{01}v=\frac{H^h_{10}u+H^h_{01}v}{h^2}
		=\frac{1}{h^2}\sum_{i=1}^n\sum_{j=1}^nY_{ij}[u(i-i_0)+v(j-j_0)]K_h(i-i_0,j-j_0),
	\end{aligned}
\end{equation}
where $u, v$ satisfy $u^2+v^2=1$ and indicate the direction of the derivative. 
Let $T_{l,k}=\sum_{r=1}^{n}\sum_{q=1}^{n}W^h_{r-l,q-k}Y_{r,q}$,
where now
\begin{equation}
	\label{W_1st}
	W^h_{r-l,q-k}=[u(r-l)+v(q-k)]K_{h}(r-l,q-k).
\end{equation}

Then we can calculate the limit of the autocorrelation function  (details can be found in \ref{Derive_1st_corr})
\begin{equation}
	\label{eq_corr_1st}
	\begin{aligned}
		\lim_{n \to \infty}Corr(T_{l,k},T_{l+i,k+j})=&\lim_{n \to \infty}\frac{\sum_{r=i}^{n}\sum_{q=j}^{n}W^h_{r,q}W^h_{r-i,q-j}}{\sum_{r=1}^{n}\sum_{q=1}^{n}(W^h_{r,q})^2}\\
		=&\frac{\int\int[u(x-i)+v(y-j)]K_{h}(x-i,y-j)(ux+vy)K_{h}(x,y)dxdy}{\int\int(ux+vy)^2K_{h}(x,y)^2dxdy}\\
		=&(1-\frac{1}{2h^2}(iu+jv)^2)\exp(-\frac{(i^2+j^2)}{4h^2}).
	\end{aligned}
\end{equation}
For convenience, we standardize $T_{l,k}$ as
\begin{equation}
	\label{standardize_1st}
	\tilde{T}_{l,k}=\frac{\sum_{r=1}^{n}\sum_{q=1}^{n}W^h_{r-l,q-k}Y_{r,q}}{\sigma\sqrt{\sum_{r=1}^{n}\sum_{q=1}^{n}(W^h_{r-l,q-k})^2}},
\end{equation}
so the variance of $\tilde{T}_{l,k}$ is 1.
A logical approach to embedding our SSS into an array of stationary Gaussian random variables as in \citet{french2013asymptotic} is to assume $1/h=C/\sqrt{\log g}$. Under these assumptions, we have the following theorem.
\begin{theorem}
	\label{SSS_extreme_1st}
	Consider a  mean-0, variance-1 Gaussian random fields $\hat{T}_{g,i,j}$ indexed by $i,j$. If for each fixed $g$ the random field $\{\hat{T}_{g,i,j}:1\leq i,j\leq g\}$  is stationary with correlation
	\begin{equation}
		\label{eq_corr_g_1st}
		\rho_{g,i,j}=[1-\frac{C^2}{2\log g}(iu+jv)^2]\exp(-\frac{(i^2+j^2)C^2}{4\log g}),
	\end{equation}
	where $C>0$ and $u^2+v^2=1$. Then
	\begin{equation*}
		\lim_{g\rightarrow\infty}P\left\{\max_{i,j=1,...,g}\hat{T}_{g,i,j}\leq u_{g^2}(x)\right\}=\exp(-\vartheta \exp(-x)),
	\end{equation*}
	where 
	\begin{equation}
		\label{eq-vartheta-1st}
		\vartheta\leq2\Phi(C)-1.
	\end{equation}
	%and $u(x)=u_{g^2}(x)$ which is defined in Section \ref{2d_extrem}.
\end{theorem}

The proof is in Appendix \ref{proof_1st_thm}. To interpret the significance of slope structures, Theorem~\ref{SSS_extreme_1st} focuses on evaluating directional derivatives along specific directions individually. Two important directions we considered are $0$  and $90$ degrees. We essentially analyze variance in these directions by modeling two random fields $\tilde{T}^{(0)}_{l,k}$ and $\tilde{T}^{(90)}_{l,k}$. Moreover, we combine the significance from different directions and derive the following Corollary. 

For $\tilde{T}^{(0)}_{l,k}$ and $\tilde{T}^{(90)}_{l,k}$, we have the covariance of them as below:

\begin{equation*}
	\begin{aligned}
		\lim_{n \to \infty}Corr(\tilde{T}^{(0)}_{l,k},\tilde{T}^{(90)}_{l,k})&=\lim_{n \to \infty}Corr(H^h_{10},H^h_{01})\\
		&=\lim_{n \to \infty}\frac{\sum_{i=1}^n\sum_{j=1}^n(i-i_0)(j-j_0)K_h(i-i_0,j-j_0)}{\sqrt{\sum_{i=1}^n\sum_{j=1}^n(i-i_0)^2K_h(i-i_0,j-j_0)\sum_{i=1}^n\sum_{j=1}^n(j-j_0)^2K_h(i-i_0,j-j_0)}}\\
		&=\frac{\int\int(x-i_0)(y-j_0)K_h(x-i_0,y-j_0)dxdy}{\sqrt{\int\int(x-i_0)^2K_h(x-i_0,y-j_0)dxdy\int\int(y-j_0)^2K_h(x-i_0,y-j_0)dxdy}}\\
		&=0
	\end{aligned}
\end{equation*}

\begin{corollary}
	\label{corollary-1st-max}
	Consider two mean-zero, variance-one Gaussian random fields  
	\(\hat{T}^{(1)}_{g,i,j}\) and \(\hat{T}^{(2)}_{g,i,j}\), each indexed by \(1 \le i,j \le g\) and satisfying
	assumption~\eqref{eq_corr_g_1st}. Let \(u_{g^2}(x)\) be the threshold corresponding to marginal tail probability \(\alpha/2\). Then for every fixed \(x \in \mathbb{R}\),
	\[
	\liminf_{g \to \infty}
	P\left\{
	\max_{i,j = 1,\dots, g}
	\left[
	(\hat{T}^{(1)}_{g,i,j})^2 \vee (\hat{T}^{(2)}_{g,i,j})^2
	\right]
	\le u_{g^2}(x)^2
	\right\}
	\ge 4\exp(-\vartheta \exp(-x))-3,
	\]
	where the right-hand side reflects Bonferroni control of the family-wise error rate over the two directional tests.
\end{corollary}
\begin{proof}
	Let us define
	\[
	M^{(1)}_g = \max_{i,j} |\hat{T}^{(1)}_{g,i,j}|, \quad
	M^{(2)}_g = \max_{i,j} |\hat{T}^{(2)}_{g,i,j}|.
	\]
	We are interested in bounding the probability
	\[
	P\left\{
	\max_{i,j} \left[(\hat{T}^{(1)}_{g,i,j})^2 \vee (\hat{T}^{(2)}_{g,i,j})^2\right]
	\le u_{g^2}(x)^2
	\right\}
	=
	P\left\{
	M^{(1)}_g \le u_{g^2}(x), \; M^{(2)}_g \le u_{g^2}(x)
	\right\}.
	\]
	
	Applying the union bound (Bonferroni inequality), we have:
	\[
	P\left\{
	M^{(1)}_g \le u_{g^2}(x), \; M^{(2)}_g \le u_{g^2}(x)
	\right\}
	\ge
	1 - P\left\{M^{(1)}_g > u_{g^2}(x) \right\}
	- P\left\{M^{(2)}_g > u_{g^2}(x) \right\}.
	\]
	
	Under the assumptions of Theorem~\ref{SSS_extreme_1st}, the distribution of \(M^{(k)}_g\) satisfies
	\[
	\lim_{g \to \infty} P\left\{M^{(k)}_g > u_{g^2}(x) \right\} = 2-2\exp(-\vartheta \exp(-x)) \quad (k=1,2).
	\]
	Substituting back into the union bound:
	\[
	\liminf_{g \to \infty} P\left\{
	\max_{i,j} \left[(\hat{T}^{(1)}_{g,i,j})^2 \vee (\hat{T}^{(2)}_{g,i,j})^2\right]
	\le u_{g^2}(x)^2
	\right\}
	\ge 1 - 2 \cdot (2-2\exp(-\vartheta \exp(-x))) = 4\exp(-\vartheta \exp(-x))-3.
	\]
	
	This completes the proof.
\end{proof}

Corollary~\ref{corollary-1st-max} provide a convenient framework for inference that is simultaneous over directions. The following Algorithm~\ref{alg:joint_slope} outlines a procedure for implementing the joint significance test for directional slopes based on the combined statistic in Corollary~\ref{corollary-1st-max}.

\begin{algorithm}[H]
	\caption{Directional Slope Joint Significance Procedure}
	\label{alg:joint_slope}
	\begin{algorithmic}[1]
		\State \textbf{Input:} Observed data matrix $Y = \{Y_{i,j}\}$ on a 2D spatial grid; smoothing bandwidth $h$; significance threshold $\tau = 2u_{g^2}(x)^2$ from Corollary~\ref{corollary-1st-max}.
		\State Perform local linear smoothing to estimate the directional slopes $\hat{a}_{10}(x)$ and $\hat{a}_{01}(x)$ at each location.
		\State Standardize the slope estimates to obtain directional test statistics $\tilde{T}^{(0)}(x)$ and $\tilde{T}^{(90)}(x)$, using weights from \eqref{W_1st} as defined in \eqref{standardize_1st}.
		\State Compute the joint statistic:
		\[
		R(x) = (\tilde{T}^{(0)}(x))^2 \vee (\tilde{T}^{(90)}(x))^2.
		\]
		\State Identify locations where $R(x) \geq \tau$ as jointly significant.
		\State \textbf{Output:} Set of locations with significant directional slope patterns.
	\end{algorithmic}
\end{algorithm}

\subsection{Significance of Curvature}
\label{sec-extreSSS}

This section provides a principled framework to detect significant curvature while controlling for multiple testing using extreme value theory. Inspired by curvature  hypothesis tests in \citep{godtliebsen2004statistical}, we focus on the distribution of quadratic forms of the Hessian matrix (representing directional second derivatives), which are derived from \eqref{argmin}
under the null hypothesis of ``no signal".

For $u^2+v^2=1$, the $(u,v)$ directional second derivative is
\begin{equation}
	\label{direct_2nd}
	(u,v)\left(
	\begin{matrix}
		\ha_{20}&\ha_{11}\\
		\ha_{11}&\ha_{02}
	\end{matrix}
	\right)(u,v)^T
	=\frac{1}{2h^4}\sum_{i=1}^n\sum_{j=1}^nY_{ij}\{[u(i-i_0)+v(j-j_0)]^2-h^2\}K_h(i-i_0,j-j_0).
\end{equation}
Similar to section \ref{sec-extreSSS_1st}, let
\begin{equation*}
	T_{l,k}=\sum_{r=1}^{n}\sum_{q=1}^{n}W^h_{r-l,q-k}Y_{r,q},
\end{equation*} 
where now
\begin{equation}
	\label{W_2nd}
	W^h_{r-l,q-k}=\{[u(r-l)+v(q-k)]^2-h^2\}K_{h}(r-l,q-k).
\end{equation}
% and the variance of $T_{l,k}$ is 1.  %, where $x=i\Delta$ and $y=j\Delta$.
Then we can calculate the limit of the autocorrelation function (details in \ref{Derive_2nd_corr})
\begin{equation}
	\label{eq_corr}
	\begin{aligned}
		\lim_{n \to \infty}Corr(T_{l,k},T_{l+i,k+j})=&\lim_{n \to \infty}\frac{\sum_{r=i}^{n}\sum_{q=j}^{n}W^h_{r,q}W^h_{r-i,q-j}}{\sum_{r=1}^{n}\sum_{q=1}^{n}(W^h_{r,q})^2}\\
		=&\frac{\int\int\{[u(x-i)+v(y-j)]^2-h^2\}K_{h}(x-i,y-j)((ux+vy)^2-h^2)K_{h}(x,y)dxdy}{\int\int((ux+vy)^2-h^2)^2K_{h}(x,y)^2dxdy}\\
		=&[1-\frac{1}{h^2}(iu+jv)^2+\frac{1}{12h^4}(iu+jv)^4]\exp(-\frac{i^2+j^2}{4h^2}).
	\end{aligned}
\end{equation}
We again now standardize $T_{l,k}$ as~\eqref{standardize_1st} with $W^h_{r-l,q-k}$ given by~\eqref{W_2nd}.
Again setting $1/h=C/\sqrt{\log g}$, we have 
\begin{theorem}
	\label{SSS_extreme}
	Consider a set of mean-0, variance-1 Gaussian random fields $\hat{T}_{g,i,j}$ . If for each fixed $g$ the random field $\{\hat{T}_{g,i,j}:1\leq i,j\leq g\}$  is stationary with correlation
	\begin{equation}
		\label{eq_corr_g_2nd}
		\rho_{g,i,j}=[1-\frac{C^2}{\log g}(iu+jv)^2+\frac{1}{12}\frac{C^4}{(\log g)^2}(iu+jv)^4]\exp(-\frac{(i^2+j^2)C^2}{4\log g}),
	\end{equation}
	where $C,u,v>0$ and $u^2+v^2=1$. Then
	\begin{equation*}
		\lim_{g\rightarrow\infty}P\left\{\max_{i,j=1,...,g}\hat{T}_{g,i,j}\leq u_{g^2}(x)\right\}=\exp(-\vartheta \exp(-x)),
	\end{equation*}
	where 
	\begin{equation}
		\label{eq-vartheta}
		\vartheta\leq2\Phi(\frac{\sqrt{6}C}{2})-1.
	\end{equation}
\end{theorem}

The proof is in Appendix \ref{proof_2nd_thm}.

In contrast to the slope-based test, which evaluates directional derivatives at a specific angle, the curvature-based test assesses the second derivative structure across all directions to determine the significance of curvature.
Since curvature is inherently directional, we consider the projected Hessian along unit vectors $(u, v) = (\cos\theta, \sin\theta)$, where $\theta$ denotes an orientation angle. Evaluating significance along each such direction allows us to identify features such as ridges, valleys, and peaks with directional specificity.
In practice, we approximate the continuous set of directions by a finite collection of angles, such as $\Theta = \{0, \pi/12, \pi/6, \ldots, 11\pi/12\}$. Note that this selection covers all unique directions due to the symmetry property of \eqref{eq_corr}: the autocorrelation in direction $\theta$ is equivalent to that in direction $\theta + \pi$, and the correlation structure \eqref{eq_corr_g_2nd} in Theorem~\ref{SSS_extreme} remains unchanged under this transformation.

For example, when identifying peaks, both eigenvalues must be significantly negative, ensuring that the Hessian matrix is negative definite. This condition is formally expressed as
\begin{equation*}
	\label{eq_direct_curve}
	(u,v) \begin{pmatrix}
		\ha_{20} & \ha_{11} \\
		\ha_{11} & \ha_{02}
	\end{pmatrix} (u,v)^T < 0
\end{equation*}  
for all unit vectors \((u,v)\). This guarantees that the local curvature exhibits a concave structure, pointing towards the presence of a peak. 
Furthermore, we can also investigate the statistical significance of other types of curvature based on the criteria in Table~\ref{Table_curve_standard}. 
\begin{table}[htbp]
	\centering
	\caption{Criteria used to classify local curvature types based on statistical significance of the directional second derivatives from  \eqref{direct_2nd}. The last column provides the corresponding color of each curvature type.}
	\label{Table_curve_standard}
	\begin{tabular}{|c|c|c|}
		\hline
		Curvature Type &Criteria &Color    \\
		\hline
		Peak&Significantly negative for all directions& Blue\\
		\hline
		Hole&Significantly positive for all directions& Yellow\\
		\hline
		Saddle point&Both significantly positive and negative appear& Red  \\
		\hline
		Ridge&Some significantly negative and rest insignificant& Purple\\
		\hline
		Valley&Some significantly positive and rest insignificant& Orange\\
		\hline
	\end{tabular}
\end{table}

Similar to Corollary~\ref{corollary-1st-max}, we propose the following Corollary~\ref{corollary-2nd-max} to combine the significance from different directions in $\Theta$. Moreover, we can use the following Algorithm~\ref{alg:directional_curvature}  to identify regions with statistically significant curvature structure.

\begin{corollary}
	\label{corollary-2nd-max}
	Consider multiple mean-zero, variance-one Gaussian random fields  
	\(\hat{T}^{(k)}_{g,i,j}\), $k=1,2,...,N$, each indexed by \(1 \le i,j \le g\) and satisfying
	assumption~\eqref{eq_corr_g_2nd}. Let \(u_{g^2}(x)\) be the threshold corresponding to marginal tail probability \(\alpha/2\). Then for every fixed \(x \in \mathbb{R}\),
	\[
	\liminf_{g \to \infty}
	P\left\{
	\max_{i,j = 1,\dots, g}
	\left[
	\vee_{k=1}^N (\hat{T}^{(k)}_{g,i,j})^2
	\right]
	\le u_{g^2}(x)^2
	\right\}
	\ge 2N\exp(-\vartheta \exp(-x))-(2N-1),
	\]
	where the right-hand side reflects Bonferroni control of the family-wise error rate over the two directional tests.
\end{corollary}
\begin{proof}
	Similar to the proof of Corollary~\ref{corollary-1st-max}, Corollary~\ref{corollary-2nd-max} can be proved by using Bonferroni correction. 
\end{proof}

\begin{algorithm}[H]
	\caption{Curvature Structure Investigation Procedure}
	\label{alg:directional_curvature}
	\begin{algorithmic}[1]
		\State \textbf{Input:} Observed data matrix $Y = \{Y_{i,j}\}$ on a 2D spatial grid; smoothing bandwidth $h$; a finite set of angles $\Theta = \{\theta_1, \ldots, \theta_K\}$; significance threshold $u(x)$ from Corollary~\ref{corollary-2nd-max}.
		\State Perform local quadratic smoothing to obtain $\ha_{20}(x),\ha_{11}(x),\ha_{02}(x)$.
		\For{each angle $\theta_k \in \Theta$}
		\State Compute the directional curvature statistic for $(u_k, v_k) = (\cos \theta_k, \sin \theta_k)$:
		\[
		Q_k(x) = (u_k, v_k)
		\begin{pmatrix}
			\ha_{20}(x) & \ha_{11}(x) \\
			\ha_{11}(x) & \ha_{02}(x)
		\end{pmatrix}
		(u_k, v_k)^T.
		\]
		\State Standardize $Q_k(x)$ to $\tilde{T}_k(x)$ as in \eqref{standardize_1st} with weights given by $\eqref{W_2nd}$.   %{\bf write how to do standardize}
		\State  Flag locations where $\tilde{T}_k(x) \geq u(x)$ or $\tilde{T}_k(x) \leq -u(x)$ as significant in direction $\theta_k$.
		\EndFor
		\State Aggregate flagged locations across all directions to produce the final curvature significance map based on Table \ref{Table_curve_standard}.
		\State \textbf{Output:} Set of locations exhibiting significant curvatures.
	\end{algorithmic}
\end{algorithm}

\section{Simulation}
\label{sec_simulation}
In this section, we investigate the empirical performance of the proposed scale-space significance testing procedures through controlled simulations. Our goal is to validate the accuracy of the theoretical thresholds derived from the extreme value distribution theory, assess the Type I error control under the global null hypothesis, and compare the detection power against established methods. %We focus on both the first and second derivatives, evaluating the procedures across multiple smoothing bandwidths and directions to ensure robustness in detecting structural features in two-dimensional fields. 
The simulation study is divided into two parts: an evaluation of Type I error rates under the null hypothesis in Section~\ref{subsec-type1} and a qualitative power analysis through curvature detection in synthetic data in Section~\ref{subsec-power}.

\subsection{Type I Error}
\label{subsec-type1}
We evaluate the empirical performance of our proposed significance testing procedures for both first and second derivatives under the global null hypothesis $H_0: m({\bf x})=0$ for all ${\bf x}$. Specifically, we generate realizations of i.i.d.  Gaussian random variables and record the number of times the test statistics exceed the theoretical thresholds derived from the limiting extreme value distributions in Theorems~\ref{SSS_extreme_1st} and~\ref{SSS_extreme}, and Corollary~\ref{corollary-1st-max} and~\ref{corollary-2nd-max} .

To assess Type I error control, we simulate 1000 independent replicates of mean-0, variance-1 Gaussian random fields on a $280 \times 280$ spatial grid with $\Delta = 1$. To avoid boundary effects in the smoothing method, we evaluate the estimates on the  $g\times g$ central grid with $g=200$.  For each replicate, we vary the smoothing bandwidth $h \in \{2, 4, 8, 16\}$ and evaluate the corresponding test statistics along a set of directional angles and overall. The significance threshold is set at $\alpha = 0.05$. Therefore, conservative control of the error rate follows when  the proportion of exceedances stays below the nominal Type I error rate 5\% (i.e. no more than 50 out of 1000).

For the first derivative (slope) test, 
we restrict the set of tested directions to the angles \[\theta =  0, \pi/2,\] because they provide an upper bound on the magnitude of the gradient.
The number of exceedances of the theoretical threshold of Theorem~\ref{SSS_extreme_1st} for each configuration is shown in Table~\ref{sim_1st}. All observed exceedance counts are below 50, confirming that the proposed gradient-based procedure successfully controls the directional Type I error at the desired level across all tested bandwidths and directions. Table~\ref{sim_1st_corr} also provides the numbers of exceedances of Corollary~\ref{corollary-1st-max} and classical SSS, respectively. The advanced SSS counts remain conservative, i.e. all below 50 as well.

\begin{table}[htbp]
	\centering
	\caption{Number of exceedances (out of 1000 simulations)  for the first derivative (slope) significance test under the null hypothesis. Shows conservative performance. }
	\label{sim_1st}
	\begin{tabular}{|c|c|c|c|c|}
		\hline
		\multirow{2}{*}{Angle} &\multicolumn{4}{c|}{Bandwidth}    \\
		\cline{2-5}
		&2&4&8&16    \\
		\hline
		0&35&12&10&5\\
		\hline
		$\pi/2$&31&24&11&6\\
		\hline
	\end{tabular}
\end{table}

\begin{table}[htbp]
	\centering
	\caption{Pure noise experiment comparing slope implementations of advanced and classical SSS. For each bandwidth $h$, the table reports the number of replications (out of 1000) with at least one exceedance. Shows statistical validity of advanced SSS. }
	\label{sim_1st_corr}
	\begin{tabular}{|c|c|c|c|c|}
		\hline
		\multirow{2}{*}{Method} &\multicolumn{4}{c|}{Bandwidth}    \\
		\cline{2-5}
		&2&4&8&16    \\
		\hline
		Advanced SSS&34&26&16&6\\
		\hline
		Classical SSS&336&303&288&234\\
		\hline
	\end{tabular}
\end{table}

Recall from the discussion after Theorem~\ref{SSS_extreme}, for the second derivative (curvature) test, it is enough to only consider angles in to  symmetry in $(-\pi/2, \pi/2]$. Furthermore, under the null distribution of i.i.d. Gaussians, the distribution of  the test statistic in \eqref{direct_2nd} for an angle $\theta$ determined by $(u,v)$ is the same as the angle determined by $(v,u)$. This allows us to focus on angles in $(-\pi/4, \pi/4]$. In Table~\ref{sim_2nd}, we  report results for angles 
\[
\theta \in \{-\pi/6, -\pi/12, 0, \pi/12, \pi/6, \pi/4\}.
\]
Again, the number of exceedances of the theoretical threshold of Theorem~\ref{SSS_extreme} in each configuration remains below the nominal level, providing strong empirical evidence that the curvature-based significance test also maintains conservative Type I error control.

\begin{table}[htbp]
	\centering
	\caption{Second derivative (curvature)  exceedances (out of 1000 simulations)  under the null hypothesis. Again shows conservative performance. }
	\label{sim_2nd}
	\begin{tabular}{|c|c|c|c|c|}
		\hline
		\multirow{2}{*}{Angle} &\multicolumn{4}{c|}{Bandwidth}    \\
		\cline{2-5}
		&2&4&8&16    \\
		\hline
		$-\pi/6$&35&23&16&8\\
		\hline
		$-\pi/12$&33&20&14&7\\
		\hline
		0&35&12&11&3\\
		\hline
		$\pi/12$&30&16&12&4\\
		\hline
		$\pi/6$&37&19&7&3\\
		\hline
		$\pi/4$&32&17&9&5\\
		\hline
	\end{tabular}
\end{table}

\begin{table}[htbp]
	\centering
	\caption{Pure noise experiment comparing advanced and classical SSS for curvature. For each bandwidth, the  number of replications (out of 1000) with at least one exceedance is reported. Unlike the classical SSS, the advanced SSS maintains type I error control.}
	\label{sim_2nd_corr}
	\begin{tabular}{|c|c|c|c|c|}
		\hline
		\multirow{2}{*}{Method} &\multicolumn{4}{c|}{Bandwidth}    \\
		\cline{2-5}
		&2&4&8&16    \\
		\hline
		Advanced SSS&22&17&11&5\\
		\hline
		Classical SSS&229&298&298&265\\
		\hline
	\end{tabular}
\end{table}

The advanced SSS test in Algorithm~\ref{alg:directional_curvature} evaluates local  curvature structure over the angles $\Theta = \{0, \pi/6, \pi/3, \pi/2, 2\pi/3, 5\pi/6\}$ simultaneously.   Table~\ref{sim_2nd_corr} reports simulation results for the above setup for both the advanced and classical SSS. Shown is the number of simulations with at least one detection of significant curvature. In the advanced SSS, the proportions exceeding the theoretical threshold from Corollary~\ref{corollary-2nd-max} are all below 5\%, whereas the classical SSS fails to control the type I error rate.    %To further validate the inference of the advanced SSS, we repeat the experiment on 1000 independent pure-noise images (no true underlying signal) with i.i.d pixels $Y_{i,j}\sim N(0,1)$ on a \(64\times 64\) plane. %For each bandwidth $h \in \{2,4,8,16\}$, Table~\ref{tab:purenoi_type1} records whether the advanced or classical SSS discover at least one significant curvature structure. The advanced SSS yields smaller false positive rates at every bandwidth.

\subsection{Power Analysis}
\label{subsec-power}
To evaluate the detection power of advanced SSS, we perform a visual comparison with the classical SSS from \citet{godtliebsen2004statistical}. The left panel of Figure~\ref{fig:curvature_comparison} is a dataset based on the Peaks and Valleys example with relatively low additive noise from Figures 7 and 8 of \citet{godtliebsen2004statistical}. We examine both slope and curvature classification results across four levels of smoothing bandwidth $h\in\{2,4,8,16\}$.

For slope detection, Algorithm~\ref{alg:joint_slope}, the top two rows of Figure~\ref{fig:curvature_comparison} use the streamline visualization  developed in \citet{godtliebsen2004statistical} to convey the significant slope information. The green streamlines are curves following the gradient while it is statistically significant. %Starting points of the streamline are sampled randomly from all pixels with significant gradient. The streamline is extended in both positive and negative gradient directions until a non-significant gradient pixel is encountered. 
The first row shows the results of our advanced SSS, while the bottom shows the original SSS. As expected, from the conservative nature of the advanced SSS, its streamlines are somewhat fewer and shorter in length. Since all the main features are present in both analyses, power loss is seen to be minimal. 

For curvature classification, Algorithm~\ref{alg:directional_curvature}, we focus on the method's ability to identify key curvature structures such as peaks, holes, ridges, valleys, and saddle points. The third and fourth rows of Figure~\ref{fig:curvature_comparison} show the results of the advanced SSS and the original SSS, respectively.  %This gives assessment of the sensitivity and consistency of the proposed approach. 
As expected, there are fewer significant pixels in the advanced SSS due to its proper statistical size. However, all major features are present, again showing good statistical power. 

More simulated data examples, contrasting  the advanced and original SSS, can be found in Figure~\ref{SM_fig:curvature_comparison_1} - \ref{SM_fig:curvature_comparison_4} in the supplementary material. Similar conclusions follow.

\begin{figure}[htbp]
	\center{\includegraphics[width=18cm]  {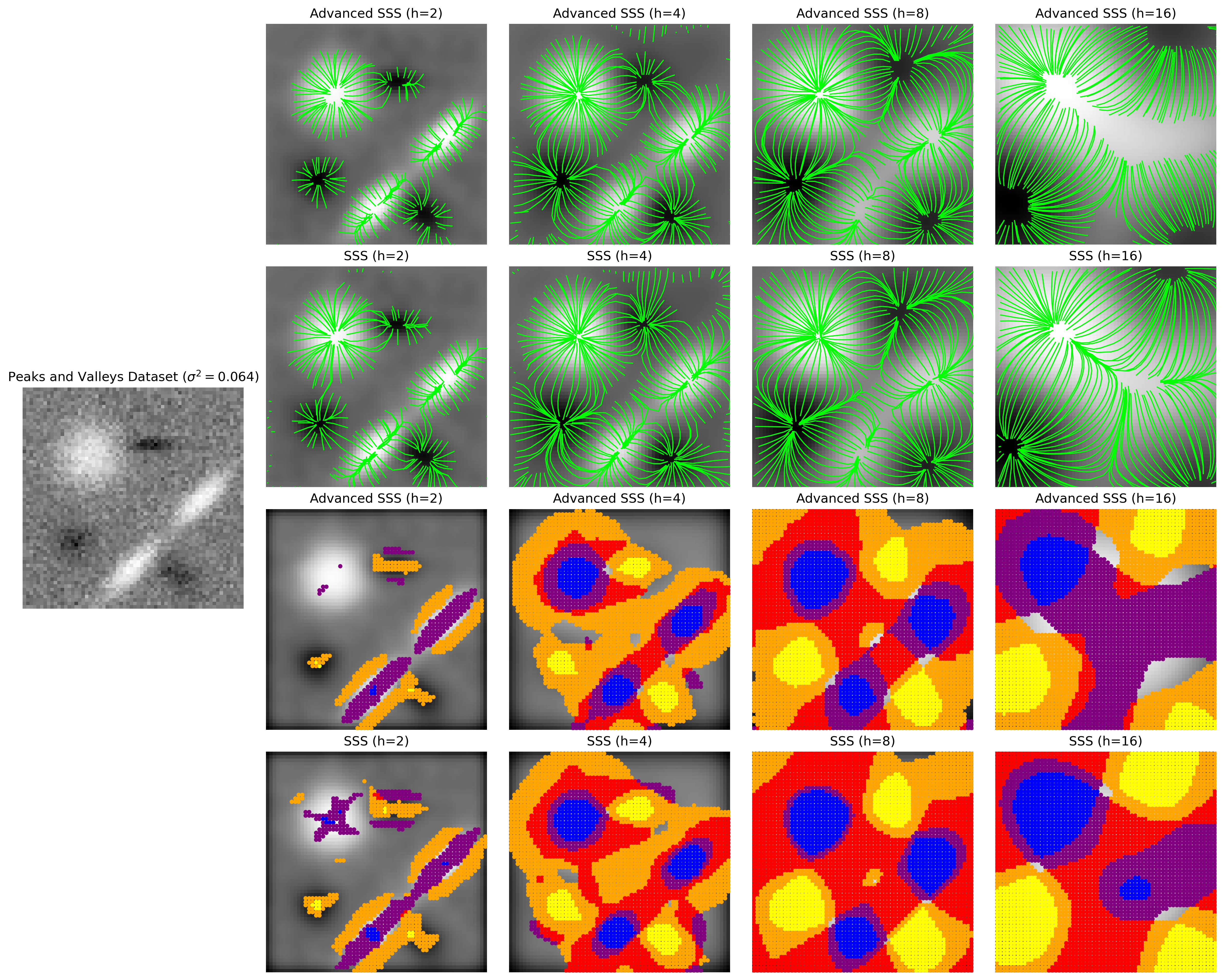}}
	\caption{
		Comparison of slope and curvature identification results from the proposed Advanced SSS method (rows 1 and 3) and the original SSS method (rows 2 and 4) across four bandwidths \( h = 2, 4, 8, 16 \). The Peaks and Valleys dataset with random noise added is shown in the leftmost panel. In the top rows, the green streamlines follow statistically significant gradients. In the bottom rows, each colored dot represents a statistically significant curvature classification: blue for peaks, yellow for holes, purple and orange for ridge and valley structures, and red for saddle points. These demonstrate comparable statistical power.
	}
	
	\label{fig:curvature_comparison}
\end{figure}

\section{Gamma Camera Analysis}
\label{sec_realdata}
As an additional comparison of the advanced SSS with the classical, we analyze a  Gamma camera phantom image from \citet{godtliebsen2004statistical}. The image records photon-count intensities on a rectangular detector array; bright uptake along rib-like structures indicates cancerous regions. We work with the \(80\times 80\) subimage extracted from a \(256\times 256\) frame (the yellow box in their Fig.~1) shown in the left panel of Figure~\ref{fig:gamma}. We examine Gaussian scale space smooths across \(h\in\{1,2,4,8\}\) (their Fig.~2). 
Rows 1 (Algorithm~\ref{alg:joint_slope}) and 3 (Algorithm~\ref{alg:directional_curvature})  of Figure~\ref{fig:gamma} show advanced SSS analyses with multiplicity control  across space and scale, with corresponding classical SSS maps shown in rows 2  and 4.

\begin{figure}[htbp]
	\center{\includegraphics[width=18cm]  {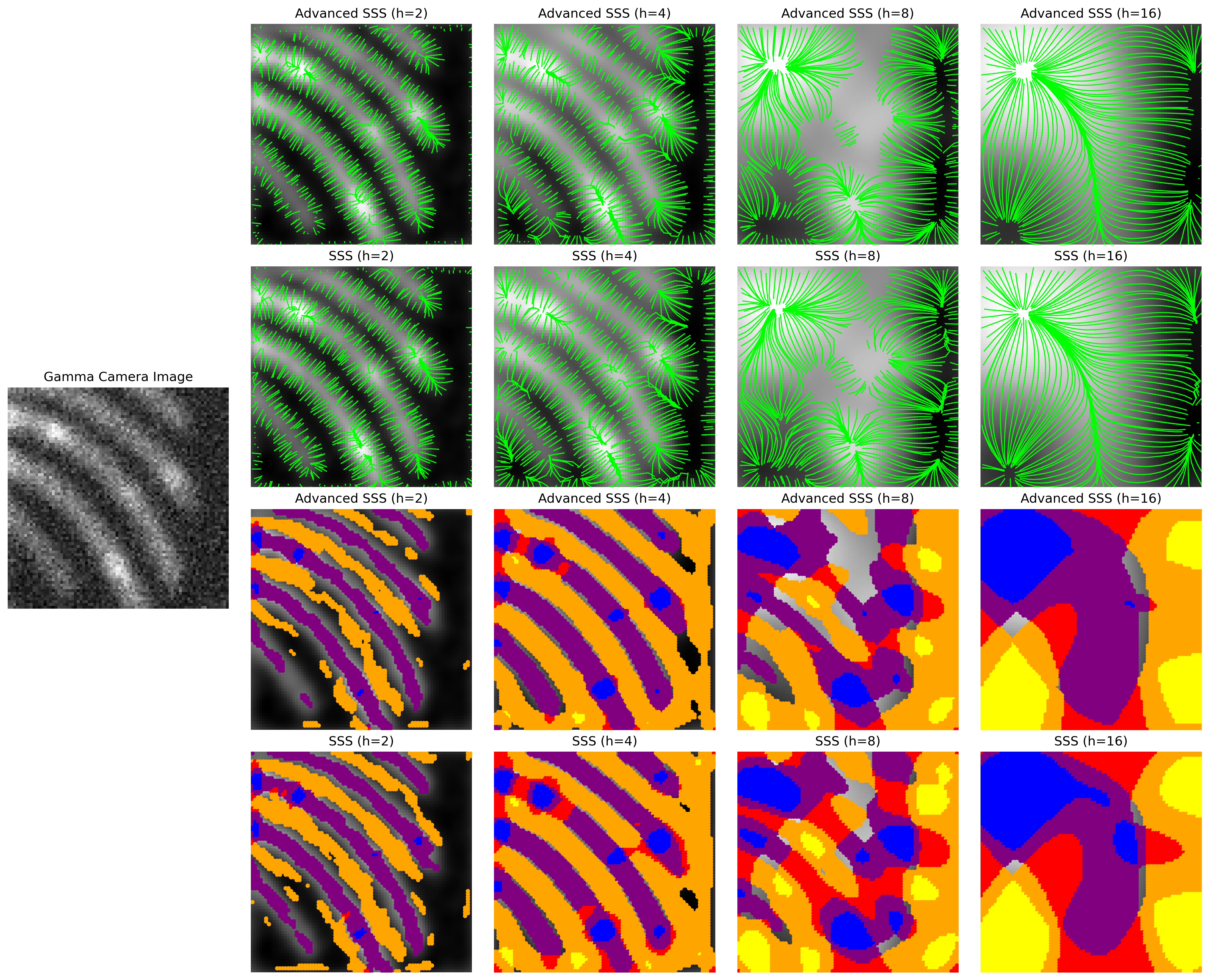}}
	\caption{Comparison of slope and curvature identification results from the proposed Advanced SSS method (rows 1 and 3) and the original SSS method (rows 2 and 4) across four bandwidths \( h = 2, 4, 8, 16 \). The Gamma camera image  is shown in the leftmost panel. In the top rows, the green streamlines follow statistically significant gradients. In the bottom rows, each colored dot represents a statistically significant curvature classification: blue for peaks, yellow for holes, purple and orange for ridge and valley structures, and red for saddle points. These demonstrate comparable statistical power.
		%Comparison of curvature identification results for Gamma camera image from the proposed advanced SSS method (top row) and the original SSS method (bottom row) across four bandwidths \( h = 2, 4, 8, 16 \). The raw image is shown in the leftmost panel for reference. Each colored dot represents a statistically significant curvature classification: blue for peaks, yellow for holes, purple and orange for ridge and valley structures, and red for saddle points. Again shows good statistical power of the advanced SSS.
	}
	
	\label{fig:gamma}
\end{figure} 

As noted in \cite{godtliebsen2004statistical}, the most relevant medical structure appears at the scale $h=4$. Ribs appear as purple ridges with orange valleys between in the curvature maps in the bottom rows. Potential cancer regions appear as blue peaks sometimes surrounded by red saddle points. The ribs are highlighted in the streamline analyses in the top rows by green lines indicating  significant gradients. The bright cancer regions are indicated by streamlines coming up the side of the rib and then bending along the ridge toward the bright white spot.  The correct multiple comparisons done by the advanced SSS results in fewer  and shorter streamlines and a somewhat smaller number of significant curvature pixels. However, all important global structures are discovered by both methods.

\section{Discussion}
\label{sec_discussion}
In this paper, we proposed an advanced distribution theory of Significance in Scale Space (SSS) \citet{godtliebsen2002significance_cluster,godtliebsen2002significance,godtliebsen2004statistical}. This advanced SSS is based on a field-wise simultaneous statistical inference framework to discover both slope and curvature structures in a 2-d image. In pure noise images  with independent standard normal entries, the advanced SSS correctly controlled type I error across bandwidths, while the classical SSS reported too many false positive features, which demonstrated invalid family-wise error control. On images with real structure, our procedure indicated salient peaks, holes and other curvature features. Our method evaluates the significance of slope or curvature at a single direction, then aggregates across multiple directions to produce a coherent curvature summary for each pixel. This two stage design controls error at the space and scale level for a fixed angle and then resolves local geometry by combining angle wise evidence into five different curvature types.

\begin{appendices}
	\section{Proofs}
	\subsection{Derivation of \eqref{eq_corr_1st}}
	\label{Derive_1st_corr}
	For denominator of Equation \ref{eq_corr_1st}, since 
	\begin{equation*}
		\begin{aligned}
			&\int\int x^2K_{h}(x,y)^2dxdy=\int\int x^2\frac{1}{4\pi^2h^4}\exp(-\frac{x^2+y^2}{h^2})dxdy\\
			=&\frac{1}{4\pi h^2}\int x^2\frac{1}{\sqrt{\pi h^2}}\exp(-\frac{x^2}{h^2})dx\int \frac{1}{\sqrt{\pi h^2}}\exp(-\frac{y^2}{h^2})dy\\
			=&\frac{1}{4\pi h^2}\frac{h^2}{2}\\
			=&\frac{1}{8\pi},
		\end{aligned}
	\end{equation*}
	we have 
	\begin{equation*}
		\begin{aligned}
			\int\int(ux+vy)^2K_{h}(x,y)^2dxdy
			=\int\int(u^2x^2+2uvxy+v^2y^2)K_{h}(x,y)^2dxdy
			=\frac{1}{8\pi}.
		\end{aligned}
	\end{equation*}
	
	For numerator of Equation \ref{eq_corr_1st}, we have 
	\begin{equation*}
		\begin{aligned}
			&\int\int[u(x-i)+v(y-j)]K_{h}(x-i,y-j)(ux+vy)K_{h}(x,y)dxdy\\
			=&\int\int[u^2x^2+2uvxy+v^2y^2-(iu+jv)(ux+vy)]K_{h}(x-i,y-j)K_{h}(x,y)dxdy\\
			=&\int\int[u^2x^2+2uvxy+v^2y^2-(iu+jv)(ux+vy)]\exp(-\frac{1}{2h^2}[(x-i)^2+x^2+(y-j)^2+y^2])dxdy.
		\end{aligned}
	\end{equation*}
	
	In order to calculate the simplified expression of this numerator, we let
	\begin{equation}
		\label{eq_calc_F}
		\begin{aligned}
			F_{i}(k)=\frac{1}{2\pi h^2}\int_{-\infty}^{\infty}x^k\exp(-\frac{1}{2h^2}[(x-i)^2+x^2])dx.
		\end{aligned}
	\end{equation}
	
	Then we calculate the above function for $k=0,1,2$.
	\begin{equation*}
		\begin{aligned}
			F_{i}(0)&=\frac{1}{2\pi h^2}\int_{-\infty}^{\infty}\exp(-\frac{1}{2h^2}[(x-i)^2+x^2])dx\\
			&=\frac{1}{2\sqrt{\pi h^2}}\exp(-\frac{i^2}{4h^2})\int_{-\infty}^{\infty}\frac{1}{\sqrt{\pi h^2}}\exp(-\frac{1}{h^2}[(x-i/2)^2])dx\\
			&=\frac{1}{2\sqrt{\pi h^2}}\exp(-\frac{i^2}{4h^2}),
		\end{aligned}
	\end{equation*}
	\begin{equation*}
		\begin{aligned}
			F_{i}(1)&=\frac{1}{2\pi h^2}\int_{-\infty}^{\infty}x\exp(-\frac{1}{2h^2}[(x-i)^2+x^2])dx\\
			&=\frac{1}{2\sqrt{\pi h^2}}\exp(-\frac{i^2}{4h^2})\int_{-\infty}^{\infty}x\frac{1}{\sqrt{\pi h^2}}\exp(-\frac{1}{h^2}[(x-i/2)^2])dx\\
			&=\frac{i}{2}F_{i}(0),
		\end{aligned}
	\end{equation*}
	\begin{equation*}
		\begin{aligned}
			F_{i}(2)&=\frac{1}{2\pi h^2}\int_{-\infty}^{\infty}x^2\exp(-\frac{1}{2h^2}[(x-i)^2+x^2])dx\\
			&=\frac{1}{2\sqrt{\pi h^2}}\exp(-\frac{i^2}{4h^2})\int_{-\infty}^{\infty}x^2\frac{1}{\sqrt{\pi h^2}}\exp(-\frac{1}{h^2}[(x-i/2)^2])dx\\
			&=(\frac{h^2}{2}+\frac{i^2}{4})F_{i}(0).
		\end{aligned}
	\end{equation*}
	
	Therefore, we have numerator of Equation \eqref{eq_corr_1st} as
	\begin{equation*}
		\begin{aligned}
			&\int\int[u^2x^2+2uvxy+v^2y^2-(iu+jv)(ux+vy)]\exp(-\frac{1}{2h^2}[(x-i)^2+x^2+(y-j)^2+y^2])dxdy\\
			=&u^2F_{i}(2)F_{j}(0)+2uvF_{i}(1)F_{j}(1)+v^2F_{i}(0)F_{j}(2)-(iu+jv)[uF_{i}(1)F_{j}(0)+vF_{i}(0)F_{j}(1)]\\
			=&F_{i}(0)F_{j}(0)[\frac{h^2}{2}+\frac{(iu+jv)^2}{4}-\frac{(iu+jv)^2}{2}]\\
			=&\frac{1}{4\pi h^2}\exp(-\frac{i^2+j^2}{4h^2})(\frac{h^2}{2}-\frac{(iu+jv)^2}{4}),
		\end{aligned}
	\end{equation*}
	and 
	\begin{equation*}
		%\label{eq_corr_value_1st}
		\begin{aligned}
			&Corr(T_{l,k},T_{l+i,k+j})=(1-\frac{1}{2h^2}(iu+jv)^2)\exp(-\frac{(i^2+j^2)}{4h^2}).
		\end{aligned}
	\end{equation*}
	\subsection{Proof of Theorem \ref{SSS_extreme_1st}}
	\label{proof_1st_thm}
	\begin{proof}
		by using Taylor expansion of the exponential term in \eqref{eq_corr_g_1st}, we get
		\begin{equation*}
			\label{eq_rho_taylor_1st}
			\begin{aligned}
				\rho_{g,i,j}&= [1-\frac{C^2}{2\log g}(iu+jv)^2]\exp(-\frac{(i^2+j^2)C^2}{4\log g})\\&=[1-\frac{C^2}{2\log g}(iu+jv)^2][1-\frac{(i^2+j^2)C^2}{4\log g}+O((\log g)^{-2})]\\
				&=1-[\frac{(iu+jv)^2}{2}+\frac{i^2+j^2}{4}]\frac{C^2}{\log g}+O((\log g)^{-2}).
			\end{aligned}
		\end{equation*}
		Therefore, we have
		\begin{equation*}
			\begin{aligned}
				\lim_{g\rightarrow\infty}(1-\rho_{g,i,j})\log g=[\frac{(iu+jv)^2}{2}+\frac{i^2+j^2}{4}]C^2,
			\end{aligned}
		\end{equation*}
		which satisfies the first condition \eqref{2d_assp1.2} of the Theorem \ref{2d_thm2.2}, i.e. $\delta_{i,j}=[\frac{(iu+jv)^2}{2}+\frac{i^2+j^2}{4}]C^2$.
		
		To verify the other conditions, we let $l_g=(\log g)^{1/2}\log(\log g)$ which satisfies \eqref{2d_2.1}. Since $f(x)=|(1-x)\exp(-x/2)|$ is a positive decreasing function when $x\geq 3$, then we have
		\begin{equation*}
			\begin{aligned}
				&\lim_{g \to \infty} \sup_{\sqrt{i^2+j^2} \geq l_g} |\rho_{g,i,j}| \log g\\
				=&\lim_{g \to \infty} \sup_{\sqrt{i^2+j^2} \geq l_g}|[1-\frac{C^2}{2\log g}(iu+jv)^2]\exp(-\frac{(i^2+j^2)C^2}{4\log g})|\log g\\
				=&\lim_{g \to \infty} \sup_{\sqrt{i^2+j^2} \geq l_g}I_{\frac{C^2}{2\log g}(iu+jv)^2<1}[1-\frac{C^2}{2\log g}(iu+jv)^2]\exp(-\frac{(i^2+j^2)C^2}{4\log g})\log g+\\
				&\qquad I_{\frac{C^2}{2\log g}(iu+jv)^2\geq 1}[\frac{C^2}{2\log g}(iu+jv)^2-1]\exp(-\frac{(i^2+j^2)C^2}{4\log g})\log g\\
				\leq&\lim_{g \to \infty} \sup_{\sqrt{i^2+j^2} \geq l_g}I_{\frac{C^2}{2\log g}(iu+jv)^2<1}\exp(-\frac{(i^2+j^2)C^2}{4\log g})\log g+\\
				&\qquad I_{\frac{C^2}{2\log g}(iu+jv)^2\geq 1}[\frac{C^2}{2\log g}(i^2+j^2)-1]\exp(-\frac{(i^2+j^2)C^2}{4\log g})\log g\\
				\leq&\lim_{g \to \infty} \exp(-\frac{l_g^2C^2}{4\log g})\log g+ I_{\frac{C^2}{2\log g}(iu+jv)^2\geq 1}[\frac{C^2}{2\log g}l_g^2-1]\exp(-\frac{l_g^2C^2}{4\log g})\log g\\
				&(\text{Since }\frac{C^2}{2\log g}l_g^2\rightarrow\infty>3\text { as } g\rightarrow\infty\text{, the second part reaches its upper bound at }i^2+j^2= l_g^2) \\
				\leq&\lim_{g \to \infty}(\log g)^{1-\log(\log g)C^2}+\frac{(C^2\log(\log g)^2/2-1)\log g}{(\log g)^{\log(\log g)C^2}} \\
				=&0,
			\end{aligned} 
		\end{equation*}
		and the Condition \ref{2d_2.5} has been verified.
		
		To verify the last condition of Theorem \ref{2d_thm2.2} (condition \ref{2d_2.6}), we need to derive the bound of $\rho_{g,i,j}$. From Equation \ref{eq_corr_g_1st}, we can get
		\begin{equation*}
			\begin{aligned}
				\rho_{g,i,j}&=[1-\frac{C^2}{2\log g}(iu+jv)^2]\exp(-\frac{(i^2+j^2)C^2}{4\log g})\\
				&\geq-\frac{C^2}{2\log g}(iu+jv)^2\exp(-\frac{(i^2+j^2)C^2}{4\log g})\\
				&\geq-\frac{C^2}{2\log g}(i^2+j^2)(u^2+v^2)\exp(-\frac{(i^2+j^2)C^2}{4\log g})\text{ (using Cauchy–Schwarz inequality)}\\
				&=\frac{C^2}{2\log g}(i^2+j^2)\exp(-\frac{(i^2+j^2)C^2}{4\log g})\\
				&=-\frac{1}{2}x\exp(-\frac{x}{4})\text{ (Let }x=\frac{C^2}{\log g}(i^2+j^2))\\
				&\geq -2e^{-1}>-1,
			\end{aligned} 
		\end{equation*}
		and
		\begin{equation*}
			\begin{aligned}
				\rho_{g,i,j}&=(1-\frac{C^2}{2\log g}(iu+jv)^2)\exp(-\frac{(i^2+j^2)C^2}{4\log g})\\
				&=(1-\frac{C^2}{2\log g}(iu+jv)^2)\exp(-\frac{[(iu+jv)^2+(iv-ju)^2]C^2}{4\log g})\\
				&\leq|1-\frac{C^2}{2\log g}(iu+jv)^2|\exp(-\frac{(iu+jv)^2C^2}{4\log g})\\
				&=|1-y/2|\exp(-\frac{y}{4})\text{ (Let }y=\frac{C^2}{\log g}(iu+jv)^2).
			\end{aligned} 
		\end{equation*}
		The curve of $f_1(y)=|1-y/2|\exp(-\frac{y}{4})$ is as the left panel of Figure \ref{f(y)} which suggests that there exist a small fixed positive $\epsilon_1$ and $\delta_y$ such that $f(y)\leq\delta_y$ if $y\geq\epsilon_1$, where $\delta_y<1$ is a constant. 
		\begin{figure}[htbp]
			\center{\includegraphics[width=14cm]  {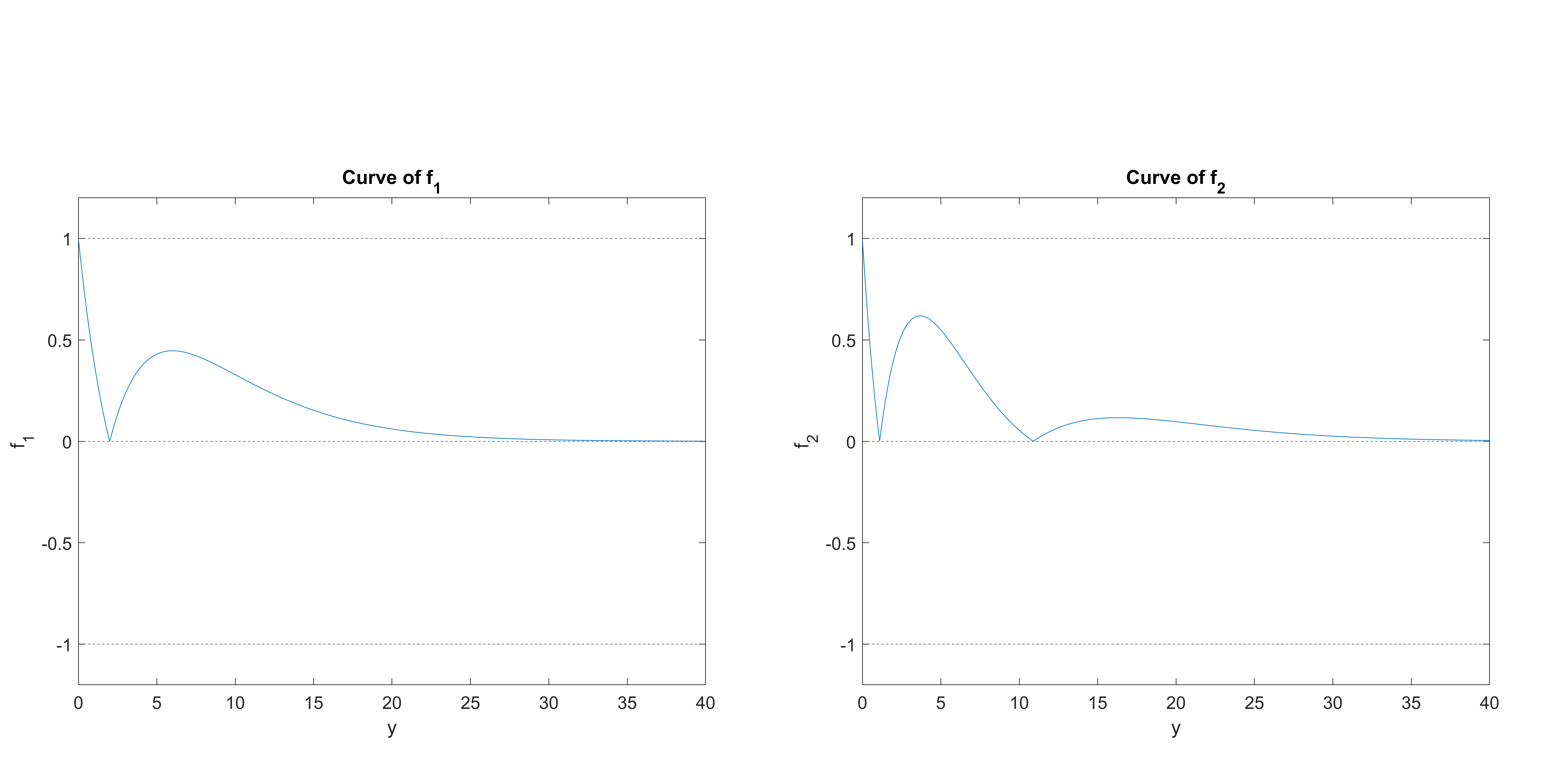}}
			\caption{Curve of function $f_1$ and $f_2$.}
			\label{f(y)}
		\end{figure} 
		
		If $\exp(-\frac{(iv-ju)^2C^2}{4\log g})\leq\delta_y$, i.e. $\frac{(iv-ju)^2C^2}{\log g}>-4\log\delta_y$,  we also have 
		\begin{equation*}
			\begin{aligned}
				\rho_{g,i,j}&=f(y)\exp(-\frac{(iv-ju)^2C^2}{4\log g})\leq\exp(-\frac{(iv-ju)^2C^2}{4\log g})\leq\delta_y.
			\end{aligned} 
		\end{equation*}
		Then 
		\begin{equation}
			\label{eq_bound1_1st}
			\begin{aligned}
				\eta_1\leq\rho_{g,i,j}\leq\delta_y,
			\end{aligned} 
		\end{equation}
		for convenience, we let $\eta_1=-2e^{-1}$. 
		
		On the other hand, if $\frac{C^2}{\log g}(iu+jv)^2=y<\epsilon_1$ and  $\exp(-\frac{(iv-ju)^2C^2}{4\log g})>\delta_y$, we can use Taylor expansion with Lagrange remainder on Equation \eqref{eq_corr_g_1st} and have
		\begin{equation*}
			\begin{aligned}
				\rho_{g,i,j}&= [1-\frac{C^2}{2\log g}(iu+jv)^2][1-\frac{(i^2+j^2)C^2}{4\log g}+\frac{(i^2+j^2)^2C^4}{32(\log g)^2}\exp(-\xi_1)]\\
				&\geq 1- [2(iu+jv)^2+(i^2+j^2)]\frac{C^2}{4\log g},
			\end{aligned}
		\end{equation*}
		where $\xi_1\in(0,\frac{(i^2+j^2)C^2}{4\log g})$ (when $g\rightarrow\infty$, $\frac{(i^2+j^2)C^2}{4\log g}>0$ clearly). Then we have $1-\rho_{g,i,j}\leq[2(iu+jv)^2+(i^2+j^2)]\frac{C^2}{4\log g}$. 
		\begin{equation*}
			\begin{aligned}
				\rho_{g,i,j}&= [1-\frac{C^2}{2\log g}(iu+jv)^2][1-\frac{(i^2+j^2)C^2}{4\log g}+\frac{(i^2+j^2)^2C^4}{32(\log g)^2}-\frac{(i^2+j^2)^3C^6}{384(\log g)^2}\exp(-\xi_2)]\\
				&\leq 1-[2(iu+jv)^2+(i^2+j^2)]\frac{C^2}{4\log g}+\left\{4(iu+jv)^2(i^2+j^2)+(i^2+j^2)^2\right\}\frac{C^4}{32(\log g)^2}\\
				&\leq 1-[2(iu+jv)^2+(i^2+j^2)]\frac{C^2}{4\log g}+\left\{2(iu+jv)^2(i^2+j^2)+(i^2+j^2)^2\right\}\frac{C^4}{16(\log g)^2}\\
				&=1-[2(iu+jv)^2+(i^2+j^2)]\frac{C^2}{4\log g}[1-(i^2+j^2)\frac{C^2}{4\log g}]\\
				&<1-[2(iu+jv)^2+(i^2+j^2)]\frac{C^2}{4\log g}\{1-\epsilon_1/4+\log \delta_y\}
			\end{aligned}
		\end{equation*}
		where $\xi_2\in(0,\frac{(i^2+j^2)C^2}{4\log g})$. Let $\epsilon^*_1=\epsilon_1/4-\log \delta_y$, then we will have $1-\rho_{g,i,j}\geq[2(iu+jv)^2+(i^2+j^2)]\frac{C^2}{4\log g}(1-\epsilon^*)$. We can set a small fixed $\epsilon$ like  $\epsilon_1=0.1$, then $\delta_y=f_1(0.1)\approx0.88$ and  $\epsilon^*\approx0.1<1$. Therefore, we have 
		\begin{equation}
			\label{eq_bound2_1st}
			\begin{aligned}
				0<[2(iu+jv)^2+(i^2+j^2)]\frac{C^2}{4\log g}(1-\epsilon^*)\leq1-\rho_{g,i,j}\leq[2(iu+jv)^2+(i^2+j^2)]\frac{C^2}{4\log g}.
			\end{aligned}
		\end{equation}
		Thus, we combine the two bounds from Equation \eqref{eq_bound1_1st} and \eqref{eq_bound2_1st} and verify condition \ref{2d_2.5} as the following.
		\begin{equation*}
			\begin{aligned}
				&\sum_{(i,j) \in \{0,1,...,l_g\}^2 \setminus \{0,1,...,m\}^2} g^{-2\frac{1-\rho_{g,i,j}}{1+\rho_{g,i,j}}} \frac{\left( \log g \right)^{-\frac{\rho_{g,i,j}}{1+\rho_{g,i,j}}}}{\sqrt{1-\rho^2_{g,i,j}}} \\
				&\leq\sum_{(i,j) \in \{0,1,...,l_g\}^2 \setminus \{0,1,...,m\}^2} g^{-2\frac{1-\delta_y}{1+\delta_y}}\frac{\left( \log g \right)^{-\frac{\eta_1}{1+\eta_1}}}{\sqrt{(1-\delta_y)(1+\eta_1)}}\\
				&\quad+\exp\left(-[2(iu+jv)^2+(i^2+j^2)]\frac{C^2}{4}(1-\epsilon^*) \right) \frac{\left( \log g \right)^{\frac{1-\rho_{g,i,j}}{2(1+\rho_{g,i,j})}}}{\sqrt{(1-\rho^2_{g,i,j})\log g}}   \\
				&\leq l_g^2 g^{-2\frac{1-\delta_y}{1+\delta_y}}\frac{\left( \log g \right)^{-\frac{\eta_1}{1+\eta_1}}}{\sqrt{(1-\delta_y)(1+\eta_1)}}+\\
				&\sum_{(i,j) \in \{0,1,...,l_g\}^2 \setminus \{0,1,...,m\}^2} \exp\left(-[2(iu+jv)^2+(i^2+j^2)]\frac{C^2}{4}(1-\epsilon^*) \right) \frac{\exp\left( \log(\log g) {[(iu+jv)^2+\frac{i^2+j^2}{4}]\frac{C^2}{2\log g}}\right)}{\sqrt{[2(iu+jv)^2+(i^2+j^2)]\frac{C^2}{4}(1-\epsilon^*)}}\\
				&=\log g (\log(\log g))^2g^{-2\frac{1-\delta_y}{1+\delta_y}}\frac{\left( \log g \right)^{-\frac{\eta_1}{1+\eta_1}}}{\sqrt{(1-\delta_y)(1+\eta_1)}}+\\
				&\sum_{(i,j) \in \{0,1,...,l_g\}^2 \setminus \{0,1,...,m\}^2} \frac{1}{\exp\left([2(iu+jv)^2+(i^2+j^2)]\frac{C^2}{4}(1-\epsilon^*-\frac{\log(\log g)}{2\log g}) \right)\sqrt{[2(iu+jv)^2+(i^2+j^2)]\frac{C^2}{4}(1-\epsilon^*)}}\\
				&\rightarrow 0 \text{ as }g\rightarrow\infty \text{ and } m\rightarrow\infty.
			\end{aligned}
		\end{equation*}
		The above equation converges to 0 because both parts will converge to 0 when $g\rightarrow\infty \text{ and } m\rightarrow\infty$. Therefore condition \eqref{2d_2.6} of Theorem (\ref{2d_thm2.2}) is verified. After verifying all the conditions, we utilize Theorem (\ref{2d_thm2.2}) to calculate the value of $\vartheta$,
		\begin{equation*}
			\begin{aligned}
				\vartheta=P(Y/2+\sqrt{2\delta_{i,j}}W_{i,j}\leq2\delta_{i,j}\text{ for all i,j}\geq1).
			\end{aligned} 
		\end{equation*}
		where $Y$ is a standard exponential random variable independent of ($W_{i,j}$). and the $W_{i,j}$ have a jointly normal distribution with mean 0 and covariance follows Equation \eqref{2d_2.4}.
		Then
		\begin{equation*}
			\begin{aligned}
				\vartheta&=P(Y/2+\sqrt{2\delta_{i,j}}W_{i,j}\leq2\delta_{i,j}\text{ for all i,j}\geq1)\\
				&\leq P(Y/2+\sqrt{2\delta_{1,0}}W_{1,0}\leq2\delta_{1,0})\\
				%&=E(P(Y/2\leq 2\delta_{1,0}-\sqrt{2\delta_{1,0}}Z)|Z)\\
				&= \int_{-\infty}^{\sqrt{2\delta_{1,0}}}(1-e^{2(2\delta_{1,0}-\sqrt{2\delta_{1,0}}z)})\frac{e^{-z^2/2}}{\sqrt{2\pi}} dz  \\
				%&= \Phi(\sqrt{2\delta_{1,0}})-\frac{1}{\sqrt{2\pi}}\int_{-\infty}^{\sqrt{2\delta_{1,0}}} e^{-(z-2\sqrt{2\delta_{1,0}})^2/2}dz \\
				%&= %\Phi(\sqrt{2\delta_{1,0}})-\frac{1}{\sqrt{2\pi}}\int_{-\infty}^{-\sqrt{2\delta_{1,0}}} e^{-z^2/2}dz \\
				&= \Phi(\sqrt{2\delta_{1,0}})-\Phi(-\sqrt{2\delta_{1,0}}) \\
				%&=2\Phi(\sqrt{2\delta_{1,0}})-1\\
				&=2\Phi(\sqrt{(u^2+\frac{1}{2})C^2})-1.
			\end{aligned} 
		\end{equation*}
		
		Similarly, we also have $\vartheta\leq2\Phi(\sqrt{2\delta_{0,1}})-1=2\Phi(\sqrt{(v^2+\frac{1}{2})C^2})-1$. Since $u^2+v^2=1$, we have 
		\begin{equation*}
			\begin{aligned}
				\vartheta&\leq\min\{2\Phi(\sqrt{2\delta_{1,0}})-1,2\Phi(\sqrt{2\delta_{0,1}})-1\}\\
				&=2\Phi(\sqrt{(\min(u^2,v^2)+\frac{1}{2})C^2})-1\\
				&=2\Phi(C)-1.
			\end{aligned} 
		\end{equation*}
		Therefore, from \eqref{2d_lim}, we can get the limit of the set of random fields $\hat{T}_{g,i,j}$ as 
		\begin{equation*}
			\lim_{g\rightarrow\infty}P\left\{\max_{i,j=1,...,g}\hat{T}_{g,i,j}\leq u_{g^2}(x)\right\}=\exp(-\vartheta \exp(-x)),
		\end{equation*}
		where Equation \ref{eq-vartheta-1st} follows
		\begin{equation*}
			\vartheta\leq2\Phi(C)-1.
		\end{equation*}
		
	\end{proof}
	\subsection{Derivation of \eqref{eq_corr}}
	\label{Derive_2nd_corr}
	\begin{equation*}
		\begin{aligned}
			Denominator=&\int\int((ux+vy)^2-h^2)^2K_{h}(x,y)^2dxdy\\
			=&\int\int((ux+vy)^4-2(ux+vy)^2h^2+h^4)K_{h}(x,y)^2dxdy\\
			=&\int\int(u^4x^4+6u^2v^2x^2y^2+v^4y^4-2h^2u^2x^2-2h^2v^2y^2+h^4)K_{h}(x,y)^2dxdy,
		\end{aligned}
	\end{equation*}
	where we have
	\begin{equation*}
		\begin{aligned}
			&\int\int x^4K_{h}(x,y)^2dxdy=\int\int x^4\frac{1}{4\pi^2h^4}\exp(-\frac{x^2+y^2}{h^2})dxdy\\
			=&\frac{1}{4\pi h^2}\int x^4\frac{1}{\sqrt{\pi h^2}}\exp(-\frac{x^2}{h^2})dx\int \frac{1}{\sqrt{\pi h^2}}\exp(-\frac{y^2}{h^2})dy\\
			=&\frac{1}{4\pi h^2}\frac{3h^4}{4}\\
			=&\frac{3h^2}{16\pi},
		\end{aligned}
	\end{equation*}
	and
	\begin{equation*}
		\begin{aligned}
			&\int\int x^2y^2K_{h}(x,y)^2dxdy=\int\int x^2y^2\frac{1}{4\pi^2h^4}\exp(-\frac{x^2+y^2}{h^2})dxdy\\
			=&\frac{1}{4\pi h^2}\int x^2\frac{1}{\sqrt{\pi h^2}}\exp(-\frac{x^2}{h^2})dx\int y^2\frac{1}{\sqrt{\pi h^2}}\exp(-\frac{y^2}{h^2})dy\\
			=&\frac{1}{4\pi h^2}(\frac{h^2}{2})^2\\
			=&\frac{h^2}{16\pi}.
		\end{aligned}
	\end{equation*}
	Therefore, we can get the value of denominator in Equation~(\ref{eq_corr}) as below,
	\begin{equation}
		\label{eq_denom}
		\begin{aligned}
			Denominator=&u^4\frac{3h^2}{16\pi}+6u^2v^2\frac{h^2}{16\pi}+v^4\frac{3h^2}{16\pi}-2h^2u^2\frac{1}{8\pi}-2h^2v^2\frac{1}{8\pi}+\frac{h^2}{4\pi}\\
			=&\frac{h^2}{\pi}[\frac{3}{16}(u^2+v^2)^2-\frac{1}{4}u^2-\frac{1}{4}v^2+\frac{1}{4}]\\
			=&\frac{3h^2}{16\pi}.
		\end{aligned}
	\end{equation}
	For numertor, we have 
	\begin{equation*}
		\begin{aligned}
			&Numerator\\
			=&\int\int\{[u(x-i)+v(y-j)]^2-h^2\}K_{h}(x-i,y-j)\left[(ux+vy)^2-h^2\right]K_{h}(x,y)dxdy\\
			=&\int\int\left[u^2(x-i)^2+2uv(x-i)(y-j)+v^2(y-j)^2-h^2\right]\left[(ux+vy)^2-h^2\right]\\
			&\qquad K_{h}(x-i,y-j)K_{h}(x,y)dxdy\\
			=&\int\int\left[(ux+vy)^2-2(iu+jv)(ux+vy)+(ui+vj)^2-h^2 \right] \left[(ux+vy)^2-h^2\right]\\
			&\qquad\exp(-\frac{1}{2h^2}[(x-i)^2+(y-j)^2+x^2+y^2])dxdy\\
			=&\int\int\left\{ (ux+vy)^4-2(iu+jv)(ux+vy)^3+[(ui+vj)^2-2h^2](ux+vy)^2+2h^2(iu+jv)(ux+vy)-[(ui+vj)^2-h^2]h^2 \right\}\\
			&\qquad\exp(-\frac{1}{2h^2}[(x-i)^2+(y-j)^2+x^2+y^2])dxdy.
		\end{aligned}
	\end{equation*}
	
	Then we calculate \eqref{eq_calc_F} for $k=3,4$.
	In order to calculate the simplified expression of the numerator of Equation \ref{eq_corr}, let
	\begin{equation*}
		\begin{aligned}
			F_{i}(3)&=\frac{1}{2\pi h^2}\int_{-\infty}^{\infty}x^3\exp(-\frac{1}{2h^2}[(x-i)^2+x^2])dx\\
			&=\frac{1}{2\sqrt{\pi h^2}}\exp(-\frac{i^2}{4h^2})\int_{-\infty}^{\infty}x^3\frac{1}{\sqrt{\pi h^2}}\exp(-\frac{1}{h^2}[(x-\frac{i}{2})^2])dx\\
			&=\frac{1}{2\sqrt{\pi h^2}}\exp(-\frac{i^2}{4h^2})\int_{-\infty}^{\infty}[(x-\frac{i}{2})^3+\frac{3}{2}ix^2-\frac{3}{4}i^2x+\frac{i^3}{8}]\frac{1}{\sqrt{\pi h^2}}\exp(-\frac{1}{h^2}[(x-\frac{i}{2})^2])dx\\
			&=\frac{3}{2}iF_{i}(2)-\frac{3}{4}i^2F_{i}(1)+\frac{i^3}{8}F_{i}(0)\\
			&=(\frac{3ih^2}{4}+\frac{i^3}{8})F_{i}(0),
		\end{aligned}
	\end{equation*}
	\begin{equation*}
		\begin{aligned}
			F_{i}(4)&=\frac{1}{2\pi h^2}\int_{-\infty}^{\infty}x^4\exp(-\frac{1}{2h^2}[(x-i)^2+x^2])dx\\
			&=\frac{1}{2\sqrt{\pi h^2}}\exp(-\frac{i^2}{4h^2})\int_{-\infty}^{\infty}x^4\frac{1}{\sqrt{\pi h^2}}\exp(-\frac{1}{h^2}[(x-\frac{i}{2})^2])dx\\
			&=F_{i}(0)\int_{-\infty}^{\infty}[(x-\frac{i}{2})^4+2ix^3-\frac{3}{2}i^2x^2+\frac{i^3}{2}x-\frac{i^4}{16}]\frac{1}{\sqrt{\pi h^2}}\exp(-\frac{1}{h^2}[(x-\frac{i}{2})^2])dx\\
			&=\frac{3h^4}{4}F_{i}(0)+2iF_{i}(3)-\frac{3}{2}i^2F_{i}(2)+\frac{i^3}{2}F_{i}(1)-\frac{i^4}{16}F_{i}(0)\\
			&=(\frac{3h^4}{4}+\frac{3i^2h^2}{4}+\frac{i^4}{16})F_{i}(0).
		\end{aligned}
	\end{equation*}
	
	Then we can get the value of the numerator of Equation \ref{eq_corr} as below,
	\begin{equation*}
		\begin{aligned}
			&Numerator\\
			=&\int\int\left[ (ux+vy)^4-2(iu+jv)(ux+vy)^3+[(ui+vj)^2-2h^2](ux+vy)^2+2h^2(iu+jv)(ux+vy)-[(ui+vj)^2-h^2]h^2 \right]\\
			&\qquad\exp(-\frac{1}{2h^2}[(x-i)^2+(y-j)^2+x^2+y^2])dxdy\\
			=&u^4F_i(4)F_j(0)+4u^3v F_i(3)F_j(1)+6u^2v^2 F_i(2)F_j(2)+4uv^3 F_i(1)F_j(3)+ v^4F_i(0)F_j(4)\\
			&-2(iu+jv)\left\{u^3F_i(3)F_j(0)+3u^2v F_i(2)F_j(1)+3uv^2 F_i(1)F_j(2)+v^3 F_i(0)F_j(3)\right\} \\
			&+[(ui+vj)^2-2h^2]\left\{u^2F_i(2)F_j(0)+2uv F_i(1)F_j(1)+v^2 F_i(0)F_j(2)\right\} +2h^2(iu+jv)\left\{uF_i(1)F_j(0)+v F_i(0)F_j(1)\right\} \\
			&-[(ui+vj)^2-h^2]h^2F_i(0)F_j(0)\\
			=&F_i(0)F_j(0)N_{i,j,u,v}(h),
		\end{aligned}
	\end{equation*}
	let $N_{i,j,u,v}(h)$ be a function of $h$ based on $i,j,u,v$. The power of $h$ in $N_{i,j,u,v}(h)$ whose coefficients are not zero are 0, 2 and 4.
	
	The coefficient of power 4 is 
	\begin{equation*}
		\begin{aligned}
			\frac{3}{4}u^4+\frac{3}{4}v^4+6u^2v^2(\frac{1}{2})^2=\frac{3}{4}(u^2+v^2)^2=\frac{3}{4},
		\end{aligned}
	\end{equation*}
	and the coefficient of power 2 is 
	\begin{equation*}
		\begin{aligned}
			&\frac{3}{4}u^4i^2+\frac{3}{4}v^4j^2+6u^2v^2(\frac{i^2+j^2}{8})+4u^3v\frac{j}{2}\frac{3i}{4}+4uv^3\frac{i}{2}\frac{3j}{4}\\
			&-2(iu+jv)\left\{\frac{3}{4}u^3i+3u^2v\frac{j}{2}\frac{1}{2}+3uv^2\frac{i}{2}\frac{1}{2}+\frac{3}{4}v^3j\right\}+(ui+vj)^2\frac{u^2+v^2}{2}\\
			&-2\left\{\frac{u^2i^2}{4}+\frac{uvij^2}{2}+\frac{v^2j^2}{4}\right\}+(iu+jv)^2-(iu+jv)^2\\
			=&\left\{\frac{3}{4}u^4i^2+\frac{3}{4}v^4j^2+\frac{3}{4}u^2v^2(i^2+j^2)+\frac{3}{2}(u^3v+uv^3)ij-\frac{3}{2}(iu+jv)^2+\frac{1}{2}(iu+jv)^2-\frac{1}{2}(iu+jv)^2\right\}\\
			=&(\frac{3}{4}u^2i^2+\frac{3}{4}v^2j^2+\frac{3}{2}uvij-\frac{3}{2}(iu+jv)^2)\\
			=&-\frac{3}{4}(iu+jv)^2,
		\end{aligned}
	\end{equation*}
	and the coefficient of power 0 is 
	\begin{equation*}
		\begin{aligned}
			&\frac{1}{16}u^4i^4+\frac{1}{16}v^4j^4+6u^2v^2(\frac{i^2j^2}{16})+4u^3v\frac{j}{2}\frac{i^3}{8}+4uv^3\frac{i}{2}\frac{j^3}{8}\\
			&-2(iu+jv)\left\{\frac{1}{8}u^3i^3+3u^2v\frac{j}{2}\frac{i^2}{4}+3uv^2\frac{i}{2}\frac{j^2}{4}+\frac{1}{8}v^3j^3\right\}\\
			&+(ui+vj)^2\left\{\frac{u^2i^2}{4}+\frac{uvij^2}{2}+\frac{v^2j^2}{4}\right\}\\
			%=&(iu+jv)^4(\frac{1}{16}-\frac{1}{4}+\frac{1}{4})\\
			=&\frac{1}{16}(iu+jv)^4.
		\end{aligned}
	\end{equation*}
	Therefore, we can conclude that 
	\begin{equation*}
		\begin{aligned}
			N_{i,j,u,v}(h)=\frac{3}{4}h^4-\frac{3}{4}(iu+jv)^2h^2+\frac{1}{16}(iu+jv)^4
		\end{aligned},
	\end{equation*}
	and
	\begin{equation}
		\label{eq_nume}
		\begin{aligned}
			Numerator		=[\frac{3}{4}h^4-\frac{3}{4}(iu+jv)^2h^2+\frac{1}{16}(iu+jv)^4]\frac{1}{4\pi h^2}\exp(-\frac{i^2+j^2}{4h^2}).
		\end{aligned}
	\end{equation}
	Thus, from \ref{eq_denom} and \ref{eq_nume}, we can get the simplified expression of Equation \ref{eq_corr} as below
	\begin{equation*}
		\label{eq_corr_value}
		\begin{aligned}
			&Corr(T_{l,k},T_{l+i,k+j})=[1-\frac{1}{h^2}(iu+jv)^2+\frac{1}{12h^4}(iu+jv)^4]\exp(-\frac{i^2+j^2}{4h^2}).
		\end{aligned}
	\end{equation*}
	\subsection{Proof of Theorem \ref{SSS_extreme}}
	\label{proof_2nd_thm}
	\begin{proof}
		By using Taylor expansion of the exponential in \eqref{eq_corr_g_2nd}, we get
		\begin{equation}
			\label{eq_rho_taylor}
			\begin{aligned}
				\rho_{g,i,j}&= [1-\frac{C^2}{\log g}(iu+jv)^2+\frac{1}{12}\frac{C^4}{(\log g)^2}(iu+jv)^4][1-\frac{(i^2+j^2)C^2}{4\log g}+O((\log g)^{-2})]\\
				&=1-[(iu+jv)^2+\frac{i^2+j^2}{4}]\frac{C^2}{\log g}+O((\log g)^{-2}).
			\end{aligned}
		\end{equation}
		Therefore, we have
		\begin{equation*}
			\begin{aligned}
				\lim_{g\rightarrow\infty}(1-\rho_{g,i,j})\log g=[(iu+jv)^2+\frac{i^2+j^2}{4}]C^2,
			\end{aligned}
		\end{equation*}
		which satisfies the first condition of the Theorem (\ref{2d_thm2.2}), i.e. $\delta_{i,j}=[(iu+jv)^2+\frac{i^2+j^2}{4}]C^2$.
		
		Similar to the proof of Theorem \ref{SSS_extreme_1st}, we let $l_g=(\log g)^{1/2}\log(\log g)$ which satisfies \ref{2d_2.1}. From Equation \ref{eq_rho_taylor}, clearly $\rho_{g,i,j}$ is a decreasing and positive function of $i$ and $j$ when they are large enough. Thus, as $g$ is large enough, we have
		\begin{equation*}
			\begin{aligned}
				&\lim_{g \to \infty} \sup_{\sqrt{i^2+j^2} \geq l_g} |\rho_{g,i,j}| \log g\\
				=&\lim_{g \to \infty} \sup_{\sqrt{i^2+j^2} = l_g}\rho_{g,i,j}\log g\\
				=&\lim_{g \to \infty}\sup_{\sqrt{i^2+j^2} = l_g}[1-\frac{C^2}{\log g}(iu+jv)^2+\frac{1}{12}\frac{C^4}{(\log g)^2}(iu+jv)^4]\exp(-\frac{(i^2+j^2)C^2}{4\log g})\log g\\
				\leq&\lim_{g \to \infty}\sup_{\sqrt{i^2+j^2} = l_g}[1+\frac{1}{12}\frac{C^4}{(\log g)^2}(i+j)^4]\exp(-\frac{l_g^2C^2}{4\log g})\log g\\
				\leq&\lim_{g \to \infty}\sup_{\sqrt{i^2+j^2} = l_g}[1+\frac{1}{3}\frac{C^4}{(\log g)^2}(i^2+j^2)^2]\exp(-\frac{l_g^2C^2}{4\log g})\log g\\
				=&\lim_{g \to \infty}[1+\frac{1}{3}\frac{C^4}{(\log g)^2}l_g^4]\exp(-\frac{l_g^2C^2}{4\log g})\log g\\
				=&\lim_{g \to \infty}\frac{\left(1+\frac{C^4(\log(\log g))^4}{3}\right)\log g}{(\log g)^{C^2\log(\log g)/4}}\\
				=&0
			\end{aligned} 
		\end{equation*}
		and the Condition \ref{2d_2.5} has been verified.
		
		To verify the last condition of Theorem \ref{2d_thm2.2} (condition \ref{2d_2.6}), we need to derive the bound of $\rho_{g,i,j}$. From Equation \ref{eq_rho_taylor}, we can get
		\begin{equation*}
			\begin{aligned}
				\rho_{g,i,j}&=[1-\frac{C^2}{\log g}(iu+jv)^2+\frac{1}{12}\frac{C^4}{(\log g)^2}(iu+jv)^4]\exp(-\frac{(i^2+j^2)C^2}{4\log g})\\
				&\geq(\frac{1}{\sqrt{3}}-1)\frac{C^2}{\log g}(iu+jv)^2\exp(-\frac{(i^2+j^2)C^2}{4\log g})\\
				&\geq(\frac{1}{\sqrt{3}}-1)\frac{C^2}{\log g}(i^2+j^2)(u^2+v^2)\exp(-\frac{(i^2+j^2)C^2}{4\log g})\text{ (using Cauchy–Schwarz inequality)}\\
				&=(\frac{1}{\sqrt{3}}-1)\frac{C^2}{\log g}(i^2+j^2)\exp(-\frac{(i^2+j^2)C^2}{4\log g})\\
				&=(\frac{1}{\sqrt{3}}-1)x\exp(-\frac{x}{4})\text{ (Let }x=\frac{C^2}{\log g}(i^2+j^2))\\
				&\geq 4e^{-1}(\frac{1}{\sqrt{3}}-1)=-0.62\ldots>-1,
			\end{aligned} 
		\end{equation*}
		and
		\begin{equation*}
			\begin{aligned}
				\rho_{g,i,j}&=[1-\frac{C^2}{\log g}(iu+jv)^2+\frac{1}{12}\frac{C^4}{(\log g)^2}(iu+jv)^4]\exp(-\frac{(i^2+j^2)C^2}{4\log g})\\
				&=[1-\frac{C^2}{\log g}(iu+jv)^2+\frac{1}{12}\frac{C^4}{(\log g)^2}(iu+jv)^4]\exp(-\frac{(i^2+j^2)(u^2+v^2)C^2}{4\log g})\\
				&\leq[1-\frac{C^2}{\log g}(iu+jv)^2+\frac{1}{12}\frac{C^4}{(\log g)^2}(iu+jv)^4]\exp(-\frac{(iu+jv)^2C^2}{4\log g})\\
				&\quad\text{(using Cauchy–Schwarz inequality, since we are deriving its upper bound,}\\
				&\qquad\text{we only need to consider the first bracket as positive)}\\
				&=[1-y+\frac{1}{12}y^2]\exp(-\frac{y}{4})\text{ (Let }y=\frac{C^2}{\log g}(iu+jv)^2).
			\end{aligned} 
		\end{equation*}
		The curve of $f(y)=[1-y+\frac{1}{12}y^2]\exp(-\frac{y}{4})$ is as Figure \ref{f(y)} which suggests that there exist a small fixed positive $\epsilon$ and $\delta_y$ such that $f(y)\leq\delta_y$ if $y\geq\epsilon$, where $\delta_y<1$ is a constant. Then we will have
		\begin{equation}
			\label{eq_bound1}
			\begin{aligned}
				\eta\leq\rho_{g,i,j}\leq\delta_y,
			\end{aligned} 
		\end{equation}
		for convenience, we let $\eta=4e^{-1}(\frac{1}{\sqrt{3}}-1)$. 
		
		If $\frac{C^2}{\log g}(iu+jv)^2=y<\epsilon$, we can extend \ref{eq_rho_taylor} and have
		\begin{equation*}
			\begin{aligned}
				\rho_{g,i,j}&= [1-\frac{C^2}{\log g}(iu+jv)^2+\frac{1}{12}\frac{C^4}{(\log g)^2}(iu+jv)^4][1-\frac{(i^2+j^2)C^2}{4\log g}+\frac{(i^2+j^2)^2C^4}{32(\log g)^2}+O((\log g)^{-3})]\\
				&=1-[(iu+jv)^2+\frac{i^2+j^2}{4}]\frac{C^2}{\log g}+\left\{\frac{(iu+jv)^4}{12}+\frac{(iu+jv)^2(i^2+j^2)}{4}+\frac{(i^2+j^2)^2}{32}\right\}\frac{C^4}{(\log g)^2}\\
				&\quad+O((\log g)^{-3}),
			\end{aligned}
		\end{equation*}
		then clearly $1-\rho_{g,i,j}\leq[(iu+jv)^2+\frac{i^2+j^2}{4}]\frac{C^2}{\log g}$. Since both $u$ and $v$ are not zero, suppose $0<u\leq v<1$, then $i^2+j^2\leq\frac{(iu+jv)^2}{u^2}$. Next, we can get the lower bound of $1-\rho_{g,i,j}$.
		\begin{equation*}
			\begin{aligned}
				\rho_{g,i,j}&=1-[(iu+jv)^2+\frac{i^2+j^2}{4}]\frac{C^2}{\log g}+\left\{\frac{(iu+jv)^4}{12}+\frac{(iu+jv)^2(i^2+j^2)}{4}+\frac{(i^2+j^2)^2}{32}\right\}\frac{C^4}{(\log g)^2}\\
				&\quad+O((\log g)^{-3})\\
				&\leq1-[(iu+jv)^2+\frac{i^2+j^2}{4}]\frac{C^2}{\log g}+\left\{(\frac{1}{12}+\frac{s}{u^2})(iu+jv)^4+(\frac{1}{4}-s+\frac{1}{32u^2})(iu+jv)^2(i^2+j^2)\right\}\frac{C^4}{(\log g)^2}\\
				&\quad+O((\log g)^{-3})\\
			\end{aligned}
		\end{equation*}
		for any $s\in[0,\frac{1}{4}]$. When $s=\frac{\frac{11}{12}u^2+\frac{1}{8}}{4u^2+1}$, $\frac{1}{12}+\frac{s}{u^2}=4(\frac{1}{4}-s+\frac{1}{32u^2})$. Since $u^2\in(0,\frac{1}{2}]$, we have $\frac{\frac{11}{12}u^2+\frac{1}{8}}{4u^2+1}\leq\frac{7}{36}<\frac{1}{4}$. Let $\epsilon_s=\frac{1}{12}+\frac{\frac{11}{12}u^2+\frac{1}{8}}{u^2(4u^2+1)}$, then
		\begin{equation*}
			\begin{aligned}
				\rho_{g,i,j} &\leq1-[(iu+jv)^2+\frac{i^2+j^2}{4}]\frac{C^2}{\log g}+\left\{\epsilon_s(iu+jv)^4+\frac{\epsilon_s}{4}(iu+jv)^2(i^2+j^2)\right\}\frac{C^4}{(\log g)^2}+O((\log g)^{-3})\\
				&=1-[(iu+jv)^2+\frac{i^2+j^2}{4}]\frac{C^2}{\log g}\left\{1-\epsilon_s(iu+jv)^2\frac{C^2}{\log g}\right\}+O((\log g)^{-3})\\
				&<1-[(iu+jv)^2+\frac{i^2+j^2}{4}]\frac{C^2}{\log g}(1-\epsilon_s\epsilon),
			\end{aligned}
		\end{equation*}
		i.e.
		\begin{equation}
			\label{eq_bound2}
			\begin{aligned}
				[(iu+jv)^2+\frac{i^2+j^2}{4}]\frac{C^2}{\log g}(1-\epsilon_s\epsilon)\leq1-\rho_{g,i,j}\leq[(iu+jv)^2+\frac{i^2+j^2}{4}]\frac{C^2}{\log g}.
			\end{aligned}
		\end{equation}

		Thus, we combine the two bounds from Equation \ref{eq_bound1} and \ref{eq_bound2} and verify condition \ref{2d_2.5} as the following.
		\begin{equation*}
			\begin{aligned}
				&\sum_{(i,j) \in \{0,1,...,l_g\}^2 \setminus \{0,1,...,m\}^2} g^{-2\frac{1-\rho_{g,i,j}}{1+\rho_{g,i,j}}} \frac{\left( \log g \right)^{-\frac{\rho_{g,i,j}}{1+\rho_{g,i,j}}}}{\sqrt{1-\rho^2_{g,i,j}}} \\
				&\leq\sum_{(i,j) \in \{0,1,...,l_g\}^2 \setminus \{0,1,...,m\}^2} g^{-2\frac{1-\delta_y}{1+\delta_y}}\frac{\left( \log g \right)^{-\frac{\eta}{1+\eta}}}{\sqrt{(1-\delta_y)(1+\eta)}}\\
				&\quad+\exp\left(-2[(iu+jv)^2+\frac{i^2+j^2}{4}]C^2(1-\epsilon_s\epsilon) \right) \frac{\left( \log g \right)^{\frac{1-\rho_{g,i,j}}{2(1+\rho_{g,i,j})}}}{\sqrt{(1-\rho^2_{g,i,j})\log g}}   \\
				&\leq l_g^2 g^{-2\frac{1-\delta_y}{1+\delta_y}}\frac{\left( \log g \right)^{-\frac{\eta}{1+\eta}}}{\sqrt{(1-\delta_y)(1+\eta)}}+\\
				&\sum_{(i,j) \in \{0,1,...,l_g\}^2 \setminus \{0,1,...,m\}^2} \exp\left(-2[(iu+jv)^2+\frac{i^2+j^2}{4}]C^2(1-\epsilon_s\epsilon) \right) \frac{\exp\left( \log(\log g) {[(iu+jv)^2+\frac{i^2+j^2}{4}]\frac{C^2}{2\log g}}\right)}{\sqrt{[(iu+jv)^2+\frac{i^2+j^2}{4}]C^2(1-\epsilon_s\epsilon)}}\\
				&=\log g (\log(\log g))^2g^{-2\frac{1-\delta_y}{1+\delta_y}}\frac{\left( \log g \right)^{-\frac{\eta}{1+\eta}}}{\sqrt{(1-\delta_y)(1+\eta)}}+\\
				&\sum_{(i,j) \in \{0,1,...,l_g\}^2 \setminus \{0,1,...,m\}^2} \exp\left(-2[(iu+jv)^2+\frac{i^2+j^2}{4}]C^2(1-\epsilon_s\epsilon) \right) \frac{\exp\left( \log(\log g) {[(iu+jv)^2+\frac{i^2+j^2}{4}]\frac{C^2}{2\log g}}\right)}{\sqrt{[(iu+jv)^2+\frac{i^2+j^2}{4}]C^2(1-\epsilon_s\epsilon)}}\\
				&\rightarrow 0 \text{ as }g\rightarrow\infty \text{ and } m\rightarrow\infty.
			\end{aligned}
		\end{equation*}
		The above equation converges to 0 because both parts will converge to 0 when $g\rightarrow\infty \text{ and } m\rightarrow\infty$. Therefore condition \ref{2d_2.6} of Theorem (\ref{2d_thm2.2}) is verified. After verifying all the conditions, we utilize Theorem (\ref{2d_thm2.2}) to calculate the value of $\vartheta$,
		\begin{equation*}
			\begin{aligned}
				\vartheta=P(Y/2+\sqrt{2\delta_{i,j}}W_{i,j}\leq2\delta_{i,j}\text{ for all i,j}\geq1).
			\end{aligned} 
		\end{equation*}
		where $Y$ is a standard exponential random variable independent of ($W_{i,j}$). and the $W_{i,j}$ have a jointly normal distribution with mean 0 and 
		\begin{equation*}
			\begin{aligned}
				EW_{i_1,j_1}W_{i_2,j_2}=\frac{\delta_{i_1,j_1}+\delta_{i_2,j_2}-\delta_{i_1-i_2,j_1-j_2}}{2(\delta_{i_1,j_1}\delta_{i_2,j_2})^{1/2}}.
			\end{aligned} 
		\end{equation*}
		Then
		\begin{equation*}
			\begin{aligned}
				\vartheta&=P(Y/2+\sqrt{2\delta_{i,j}}W_{i,j}\leq2\delta_{i,j}\text{ for all i,j}\geq1)\\
				&\leq P(Y/2+\sqrt{2\delta_{1,0}}W_{1,0}\leq2\delta_{1,0})\\
				&=E(P(Y/2\leq 2\delta_{1,0}-\sqrt{2\delta_{1,0}}Z)|Z)\\
				&= \int_{-\infty}^{\sqrt{2\delta_{1,0}}}(1-e^{2(2\delta_{1,0}-\sqrt{2\delta_{1,0}}z)})\frac{e^{-z^2/2}}{\sqrt{2\pi}} dz  \\
				&= \Phi(\sqrt{2\delta_{1,0}})-\frac{1}{\sqrt{2\pi}}\int_{-\infty}^{\sqrt{2\delta_{1,0}}} e^{-(z-2\sqrt{2\delta_{1,0}})^2/2}dz \\
				&= \Phi(\sqrt{2\delta_{1,0}})-\frac{1}{\sqrt{2\pi}}\int_{-\infty}^{-\sqrt{2\delta_{1,0}}} e^{-z^2/2}dz \\
				&= \Phi(\sqrt{2\delta_{1,0}})-\Phi(-\sqrt{2\delta_{1,0}}) \\
				&=2\Phi(\sqrt{2\delta_{1,0}})-1\\
				&=2\Phi(\sqrt{2(u^2+\frac{1}{4})C^2})-1.
			\end{aligned} 
		\end{equation*}
		
		Similarly, we also have $\vartheta\leq2\Phi(\sqrt{2\delta_{0,1}})-1=2\Phi(\sqrt{2(v^2+\frac{1}{4})C^2})-1$. Since $u^2+v^2=1$, we have 
		\begin{equation*}
			\begin{aligned}
				\vartheta&\leq\min\{2\Phi(\sqrt{2\delta_{1,0}})-1,2\Phi(\sqrt{2\delta_{0,1}})-1\}\\
				&=2\Phi(\sqrt{2(\min(u^2,v^2)+\frac{1}{4})C^2})-1\\
				&\leq2\Phi(\sqrt{2(\frac{1}{2}+\frac{1}{4})C^2})-1\\
				&=2\Phi(\sqrt{\frac{3}{2}C^2})-1\\
				&=2\Phi(\frac{\sqrt{6}C}{2})-1
			\end{aligned} 
		\end{equation*}
		Then Equation \ref{eq-vartheta} follows.
	\end{proof}
\end{appendices}

\section*{Data Availability}
The data used in Section~\ref{subsec-power} and Section~\ref{sec_realdata} can be found in the GitHub repository: \url{https://github.com/jerryliu01998/Advanced_SSS/tree/main/Data}.

\section*{Code Availability}
All the custom code used in this research can be found in the specified GitHub repository: \url{https://github.com/jerryliu01998/Advanced_SSS}.

\section*{Acknowledgement}
The research of Rui Liu and  J. S. Marron  was partially supported by the National Science Foundation under Grants DMS-2113404 and DMS-2515765, and by a UNC Computational Medicine Program Pilot Award. Jan Hannig’s research was supported in part by the National Science Foundation under Grants DMS-2113404, DMS-2210337, and DMS-2515303, and by the United States–Israel Binational Science Foundation (BSF), Jerusalem, under Grant No. 2024055.

\bibliographystyle{plainnat}
\bibliography{Ref_SSS}

\end{document}